\documentclass[10pt, a4paper]{amsart}
 
\usepackage[T1]{fontenc}
\usepackage{microtype}

\usepackage{amsmath}								
\usepackage{amssymb}
\usepackage{amsthm}
\usepackage{amscd}
\usepackage{amsfonts}
\usepackage{bbm}
\usepackage{stmaryrd}

\usepackage{extarrows}

\usepackage[bookmarks=true, colorlinks=true, linkcolor=blue!50!black,
citecolor=orange!50!black, urlcolor=orange!50!black, pdfencoding=unicode]{hyperref}
\usepackage{color}
\usepackage{comment}
\usepackage{mathrsfs}

\usepackage{tikz}									
\usetikzlibrary{matrix}
\usetikzlibrary{patterns}
\usetikzlibrary{matrix}
\usetikzlibrary{positioning}
\usetikzlibrary{decorations.pathmorphing}
\usetikzlibrary{cd}
\usetikzlibrary{trees}

\usepackage{fullpage}
\usepackage{enumerate}

\newtheorem*{theorem*}{Theorem}

\newtheorem{maintheorem}{Theorem}[section]

\newtheorem{theorem}{Theorem}[section]
\newtheorem{lemma}[theorem]{Lemma}
\newtheorem{proposition}[theorem]{Proposition}
\newtheorem{corollary}[theorem]{Corollary} 

\theoremstyle{definition}
\newtheorem{definition}[theorem]{Definition}
\newtheorem{example}[theorem]{Example}

\newtheorem{remark}[theorem]{Remark}

\newtheoremstyle{myitemstyle}						
	{}			
	{}			
	{}			
	{}			
	{}			
	{.}			
	{ }			
	{}			
\theoremstyle{myitemstyle}
\newtheorem{myitemthm}{}

\newcommand{\martin}[1]{{\color{green!60!black} \sf Martin: [#1]}}

\definecolor{chocolate}{RGB}{143,89,2}

\newcommand{\hooklongrightarrow}{\lhook\joinrel\longrightarrow}
\newcommand{\twoheadlongrightarrow}{\relbar\joinrel\twoheadrightarrow}

\newcommand{\R}{\mathbb{R}}
\newcommand{\Rbar}{\overline{\mathbb{R}}}
\newcommand{\Z}{\mathbb{Z}}

\newcommand{\Q}{\mathbb{Q}}
\renewcommand{\C}{\mathbb{C}}

\renewcommand{\P}{\mathbb{P}}
\newcommand{\PP}{\mathbb{P}}

\newcommand{\OO}{\mathcal{O}}

\newcommand{\RR}{\mathbb{R}}
\newcommand{\A}{\mathbb{A}}
\renewcommand{\G}{\mathbb{G}}

\newcommand{\TT}{\mathbb{T}}
\newcommand{\T}{\mathbb{T}}
\newcommand{\on}{\operatorname{}}

\newcommand{\calA}{\mathcal{A}}
\newcommand{\calAbar}{\overline{\mathcal{A}}}
\newcommand{\calB}{\mathcal{B}}
\newcommand{\calBbar}{\overline{\mathcal{B}}}
\newcommand{\scrB}{\mathscr{B}}

\newcommand{\calI}{\mathcal{I}}

\newcommand{\calL}{\mathcal{L}}

\newcommand{\calN}{\mathcal{N}}
\newcommand{\calNbar}{\overline{\mathcal{N}}}

\newcommand{\calX}{\mathcal{X}}
\newcommand{\calXbar}{\overline{\mathcal{X}}}

\DeclareMathOperator{\Spec}{Spec}
\DeclareMathOperator{\Proj}{Proj}

\DeclareMathOperator{\val}{val}

\DeclareMathOperator{\Trop}{Trop}

\DeclareMathOperator{\pr}{pr}
\DeclareMathOperator{\id}{id}

\DeclareMathOperator{\GL}{GL}

\DeclareMathOperator{\PGL}{PGL}
\DeclareMathOperator{\trop}{trop}

\DeclareMathOperator{\an}{an}
\DeclareMathOperator{\univ}{univ}
\DeclareMathOperator{\Gr}{Gr}
\DeclareMathOperator{\rank}{rank}
\DeclareMathOperator{\diag}{diag}

\newcommand{\aset}[1]{\left\{ #1  \right\}}
\newcommand{\suchthat}{\big\vert \,}
\newcommand{\mbases}[1]{\mathscr{B}\left(#1  \right)}
\newcommand{\allone}{\mathbbm{1}}

\newcommand{\seminorm}[1]{\left\lVert #1  \right \rVert}

\title{Buildings, valuated matroids, and tropical linear spaces} 

\date{}

\author{Luca Battistella}
\address{Institut für Mathematik, Humboldt--Universität zu Berlin, 10099 Berlin, Germany}
\email{luca.battistella@hu-berlin.de}

\author{Kevin K\"uhn}
\address{Institut f\"ur Mathematik, Goethe--Universit\"at Frankfurt,
60325 Frankfurt am Main, Germany}
\email{kuehn@math.uni-frankfurt.de}

\author{Arne Kuhrs}
\address{Institut f\"ur Mathematik, Goethe--Universit\"at Frankfurt,
60325 Frankfurt am Main, Germany}
\email{kuhrs@math.uni-frankfurt.de}

\author{Martin Ulirsch}
\address{Institut f\"ur Mathematik, Goethe--Universit\"at Frankfurt,
60325 Frankfurt am Main, Germany}
\email{ulirsch@math.uni-frankfurt.de}

\author{Alejandro Vargas}
\address{Institut f\"ur Mathematik, Goethe--Universit\"at Frankfurt,
60325 Frankfurt am Main, Germany}
\email{alejandro@vargas.page}

\begin{document}

\maketitle

\begin{abstract}
Affine Bruhat--Tits buildings are geometric spaces extracting the combinatorics of algebraic groups. The building of $\mathrm{PGL}$ parametrizes flags of subspaces/lattices in or, equivalently, norms on a fixed finite-dimensional vector space, up to homothety. It has first been studied by Goldman and Iwahori as a piecewise-linear analogue of symmetric spaces. 
The space of seminorms compactifies the space of norms and admits a natural  surjective restriction map from the Berkovich analytification of projective space that factors the natural tropicalization map. 
Inspired by Payne's result that the analytification is the limit of all tropicalizations, we show that the space of  seminorms is the limit of all tropicalized \emph{linear} embeddings $\iota\colon\mathbb{P}^r\hookrightarrow\mathbb{P}^n$ and prove  a faithful tropicalization result for compactified linear spaces. The space of seminorms is in fact the tropical linear space associated to the  universal realizable valuated matroid. \end{abstract}

\setcounter{tocdepth}{1}
\tableofcontents


\section*{Introduction}
Let $K$ be a complete non-Archimedean field (possibly carrying the trivial absolute value). The non-Archimedean analytic approach to tropical geometry (see e.g.~\cite{BPR, EinsiedlerKapranovLind, Gubler_guide, Payne_anallimittrop}) 
makes crucial use of non-Archimedean analytic spaces in the sense of Berkovich \cite{Berkovich_book}; the particular topological properties of these Berkovich spaces let us think of the analytification $X^{\an}$ of an algebraic variety $X$ as a form of universal tropicalization that is independent of the chosen coordinate system. 
In this regard, the central result presented in \cite{Payne_anallimittrop} tells us that the analytification $X^{\an}$ of a quasi-projective algebraic variety $X$ over $K$ is the projective limit of all tropicalizations with respect to all the embeddings in toric varieties
(also see \cite{GrossFosterPayne, GiansiracusaGiansiracusa, KuronyaSouzaUlirsch} for generalizations of this result). 

In this article we argue that when considering only the tropicalization of projective spaces linearly embedded into higher-dimensional projective spaces, the role of Berkovich analytic space is played by the Goldman--Iwahori space $\calXbar_r(K)$ of homothety classes of non-trivial seminorms on $(K^{r+1})^\ast$ (see \cite{GoldmanIwahori} but also \cite{RemyThuillierWerner_survey, Werner_seminorms,RemyThuillierWernerII}).
The locus $\calBbar_r(K)$ of diagonalizable seminorms in $\calXbar_r(K)$ is  (a compactification of) the affine Bruhat--Tits building  of $\PGL_{r+1}(K)$, when $K$ carries a non-trivial valuation, and the cone over the spherical building of $\PGL_{r+1}(K)$, when the valuation on $K$ is trivial. When $K$ is spherically complete, we have $\calBbar_r(K)=\calXbar_r(K)$.

Let $\iota\colon \P^r \hookrightarrow \P^n$ be a linear closed immersion. The tropicalization $\Trop(\P^r,\iota)$ of $\P^r$ with respect to the embedding $\iota$ is the projection of  $\iota(\P^r)^{\an}\subseteq (\P^n)^{\an}$ under the natural tropicalization map
\begin{equation*}
\trop_{\P^n}\colon (\P^n)^{\an}\longrightarrow \T\P^n\\
\end{equation*}
to $\T\P^n=\Rbar^{n+1}-\big\{(\infty,\dots,\infty)\big\}/\ \R \cdot \mathbbm{1}$ that is essentially given by taking coordinate-wise valuations (see Section \ref{section_analytification&tropicalization} below for details). 
The tropicalization $\Trop(\P^r,\iota)$ is a projective tropical linear space in $\T\P^n$ and is associated to a realizable valuated matroid (see Section \ref{section_valuatedmatroids} below for details).

We shall see in Section \ref{section_analytification&tropicalization} below that there is a natural continuous and surjective restriction map $\tau\colon (\P^n)^{\an} \rightarrow  \calXbar_n(K)$ that factors the tropicalization map as 
\begin{equation*}
    (\P^n)^{\an}\xlongrightarrow{\tau} \calXbar_n(K)\xlongrightarrow{\trop_{\calXbar_n}} \T\P^n
\end{equation*}
such that the tropicalization of $\Trop(\P^r,\iota)$ is given as the projection of $\calXbar_r(K)\subseteq\calXbar_n(K)$ under $\trop_{\calXbar_n}$. 
The connection between affine buildings and Berkovich analytic spaces is well-established in the literature. We refer the reader in particular to \cite{Berkovich_book} as well as to \cite{RemyThuillierWernerI, RemyThuillierWernerII, RemyThuillierWerner_survey}. 

Denote by $I$ the category whose objects are linear closed immersions $\iota\colon \P^r\hookrightarrow U\subseteq\P^n$ into a torus-invariant open subset $U\subseteq \P^n$; a morphism between $\iota\colon \P^r\hookrightarrow U\subseteq\P^n$ and $\iota'\colon \P^r\hookrightarrow U'\subseteq\P^{n'}$ is given by a linear toric morphism $\varphi\colon U\rightarrow U'$ such that $\varphi\circ \iota=\iota'$. 
The tropicalization with respect to $\iota\colon \P^r\hookrightarrow U$ is naturally homeomorphic to the tropicalization with respect to the composition $\P^r\hookrightarrow U\subseteq \P^n$. For a toric morphism $\varphi\colon U\rightarrow U'$ such that $\varphi\circ \iota=\iota'$, we have a natural induced map $\varphi^{\trop}\colon U^{\trop}\rightarrow U'^{\trop}$ such that $\varphi^{\trop}\big(\Trop(\P^r,\iota)\big)\subseteq \Trop(\P^r,\iota')$, making $I$ into a cofiltered category. Requiring morphisms only to be defined on an open subset of projective space provides us with the added flexibility that we need in the following.

\begin{maintheorem}\label{mainthm_limits}
The tropicalization maps induce a natural homeomorphism 
\begin{equation*}
\overline{\calX}_r(K)\xlongrightarrow{\sim}\varprojlim_{\iota \in I} \Trop\big(\PP^r,\iota\big) \ ,
\end{equation*}
where the projective limit is taken over the category $I$. 
\end{maintheorem}

Let $I'$ be the full subcategory of $I$, whose objects are linear embeddings $\iota\colon \P^r\hookrightarrow U\subseteq\P^n$ whose images meet the big torus $\G_m^n\subseteq \P^n$. Then Theorem \ref{mainthm_limits} provides us with a homeomorphism between the 
space of norms $\calX_r(K)$ and the projective limit of all non-compactified tropicalizations $\Trop\big(\iota(\P^r)\cap\G_m^n\big)$. When $K$ is spherically complete, the Goldman--Iwahori space $\calX_r(K)$ is equal to  $\calB_r(K)$, and therefore we have a natural homeomorphism 
\begin{equation*}
\calB_r(K)\xlongrightarrow{\sim}\varprojlim_{\iota \in I'} \Trop\big(\iota(\P^r)\cap\G_m^n\big) \ . 
\end{equation*}

In addition to Theorem \ref{mainthm_limits} we also prove the following, 
which one may think of as a new addition to the literature on faithful tropicalization (see e.g.~\cite{BPR, CuetoWernerHaebich, GRWI,GRWII} for other instances).

\begin{maintheorem}\label{mainthm_section}
Let $\iota\colon\P^r\hookrightarrow \P^n$ be a  linear closed immersion. Then there is a natural piecewise linear embedding $J\colon \Trop(\P^r,\iota)\rightarrow \calBbar_r(K)$ that makes the following diagram commute
\begin{equation*}
\begin{tikzcd}
\calBbar_r(K)\arrow[rd,"\calBbar(\iota)"']&&\Trop(\P^r,\iota) \arrow[rd, "\subseteq"]\arrow[ll, "J"']&\\
& \calBbar_n(K)\arrow[rr,"\trop"]&& \T\P^n.
\end{tikzcd}
\end{equation*}
\end{maintheorem}

Suppose now that $K$ is discretely valued. Then, for one, the section in Theorem \ref{mainthm_section} recovers the so-called \emph{membrane} of a realization of $\Trop(\P^r,\iota)$ from \cite[Lemma 4]{JoswigSturmfelsYu}. If  $K$ is also local, there is a natural embedding of $\calBbar_r(K)=\calXbar_r(K)$ into the Berkovich analytic space $(\P^r)^{\an}$ (see \cite[Section 3]{RemyThuillierWernerII} for details). Theorem \ref{mainthm_limits} then tells us that the collection of all linear reembeddings $\P^r\hookrightarrow \P^n$ recovers exactly $\calBbar_r(K)=\calXbar_r(K)$. The main result of \cite{Payne_anallimittrop}, on the other hand, tells us that, once we also allow non-linear algebraic re-embeddings of $\P^r$ into suitable toric varieties, we recover the whole Berkovich analytification of $\P^r$. 

Theorem \ref{mainthm_limits} and Theorem \ref{mainthm_section} together provide us with a heuristic saying that $\calXbar_r(K)$ is in some sense a \emph{universal (realizable) tropical linear space}. This goes hand in hand with the work of Dress and Terhalle \cite{DressTerhalle_treeoflife}, in which the authors realize the building $\calB_r(K)$ as the tight span of a suitable infinite valuated matroid, see Section \ref{section_tightspan} below.

In Section \ref{section_infiniteGrassmannian} we build on this observation and construct a tropical linear space for any infinite valuated matroid (expanding on the finite case, see \cite{Speyer, SpeyerSturmfels} as well as \cite[Chapter 4]{MaclaganSturmfels}). 
This allows us to make precise the idea that the Goldman-Iwahori space is the tropical linear space of the \emph{universal realizable valuated matroid} $w_{\univ}$ that is given by the map 
\begin{equation*}
w_{\univ}\colon{(K^{r+1})^{\ast} \choose r+1}\longrightarrow \Rbar
\end{equation*}
induced by the permutation-invariant map $\val\circ\det\colon K^{(r+1)\times (r+1)}\rightarrow \Rbar$.

\begin{maintheorem}\label{maintheorem_infinite_tropicalization}
Let $w_{\univ}$ be the universal realizable valuated matroid. Then
$$\calXbar_r(K)=\Trop(w_{\univ}).$$
\end{maintheorem}

Our approach, in particular, provides us with a notion of tropicalization for a finite-dimensional linear space embedded into an infinite-dimensional space, see again Section \ref{section_infiniteGrassmannian} for details. Theorem \ref{maintheorem_infinite_tropicalization} is also, in spirit, very similar to the universal (possibly non-linear) tropicalization of \cite{GiansiracusaGiansiracusa}, which may be identified with Berkovich analytic space.

It is well-known that Bruhat--Tits buildings also admits a simplicial structure with can be described in terms  of filtrations by linear subspaces, when $K$ carries the trivial absolute value, and by sublattices, when $K$ carries a non-trivial discrete absolute value. 
We incorporate those perspectives into our story in Examples \ref{ex_trivialvaluationdescription} and \ref{B_1(K)trivial}, while, in Section \ref{section_trivvalmatroid=matroid}, we recall the analogous story how valuated matroids over a field with trivial absolute value are nothing but matroids. 
In Section \ref{section_examples} we  illustrate our main results using this alternative point of view on buildings.

\subsection*{Conventions} We write $\Rbar$ for $\R\cup\{\infty\}$ with tropical operations $\min$ as well as $+$, and $\T\P^n$ for $\Rbar^{n+1}\setminus\{(\infty,\ldots,\infty)\}/\sim$, where $\sim$ is the equivalence relation generated by translation by a real multiple of the vector $\mathbbm{1}=(1,\ldots,1)$. We write $[n]$ for the set $\{0,\ldots,n\}$ and ${E \choose r}$ for the set of subsets of cardinality $r$ of a given set $E$. By $K$ we denote a complete non-Archimedean field, with ring of integers $\OO_K$ and residue field $k$; in the discretely valued case, $\OO_K$ is Noetherian and its maximal ideal is generated by one element, called a uniformizer and denoted by $\pi$.

\subsection*{Acknowledgements}
This collaboration was initiated while we were preparing for and participating in a learning seminar, the GAUS-AG on buildings; we thank all other participants, in particular Jiaming Chen, Andreas Gross, Johannes Horn, Lucas Gerth, J\'er\^ome Poineau, Felix R\"ohrle, Pedro Souza, and Jakob Stix. Particular thanks are due to Annette Werner for answering our countless questions about buildings on several occasions. We also thank Jeffrey Giansiracusa, Marvin Hahn, Diane Maclagan, Anne Parreau, Bertrand R\'emy, Victoria Schleis, Petra Schwer, and Jonathan Wise for helpful conversations at various (pre-)stages of this project.

\subsection*{Funding} This project has received funding from the Deutsche Forschungsgemeinschaft  (DFG, German Research Foundation) TRR 326 \emph{Geometry and Arithmetic of Uniformized Structures}, project number 444845124, from the Deutsche Forschungsgemeinschaft (DFG, German Research Foundation) Sachbeihilfe \emph{From Riemann surfaces to tropical curves (and back again)}, project number 456557832, and from the LOEWE grant \emph{Uniformized Structures in Algebra and Geometry}.
L.B. has received funding from the ERC Advanced Grant SYZYGY of the European Research Council (ERC) under the European Union Horizon 2020 research and innovation program (grant agreement No. 834172).
A.V. has received funding from the Swiss National Science Foundation, grant number 200142.


\section{The affine building and its compactification}

Let $K$ be a complete non-Archimedean field. In this section, we recall some fundamentals about the space of  seminorms on a finite-dimensional vector space $V$ over $K$. The (compactification of the) affine building of $\PGL_n$ is the a priori proper subset of diagonalizable (semi)norms. If $K$ is spherically complete these two spaces will turn out to be the same.

Let $V$ be a vector space over $K$ of dimension $n$. 
\begin{definition}
A \emph{norm} on $V$ is a map $\vert\vert.\vert\vert\colon V\rightarrow \R$ that fulfills the following axioms:
\begin{enumerate}[(i)]
\item For all $v\in V$ we have $\vert\vert v\vert\vert \geq 0$ and $\vert\vert v\vert\vert=0$ if and only if $v=0$;
\item For all $v\in V$ and $\lambda\in K$ we have 
\begin{equation*}
\vert\vert \lambda \cdot v\vert\vert =\vert \lambda \vert \cdot \vert\vert v\vert\vert \ .
\end{equation*}
\item For all $v,w\in V$  the strong triangle inequality
\begin{equation*}
\vert\vert v+w\vert\vert \leq \max\big\{\vert\vert v\vert\vert,\vert\vert w\vert\vert\big\}
\end{equation*}
holds.
\end{enumerate}
If in (i) we only require $\vert\vert v\vert\vert\geq 0$ and allow vectors $v\in V-\{0\}$ with $\vert\vert v\vert\vert=0$ we say that $\vert\vert.\vert\vert$ is a \emph{seminorm}. A seminorm $\vert\vert.\vert\vert$ is said to be \emph{non-trivial} if there is a $v\in V$ such that $\vert\vert v\vert\vert\neq 0$.

\end{definition}

\begin{example}
Pick a basis $\mathbf{e}=(e_1,\ldots, e_n)$ of $V$ and $\vec{a}=(a_1,\ldots, a_n)\in \Rbar^n$. We may associate to this datum a seminorm $\vert\vert.\vert\vert_{\mathbf{e},\vec{a}}$ on $V$ given by associating to $v=\sum_{i=1}^n \lambda_i e_i$ the value $\max_{i=1,\ldots, n}\big\{\vert\lambda_i\vert e^{-a_i}\big\}$. The seminorm is non-trivial if and only if at least one $a_i\neq \infty$ and it is a norm if and only if all $a_i\neq \infty$. 
\end{example}

A seminorm of the form $\vert\vert.\vert\vert_{\mathbf{e},\vec{a}}$ for a basis $\mathbf{e}$ and coefficients $\vec{a}\in\Rbar^n$ is said to be \emph{diagonalizable}.

\begin{definition}
    A non-Archimedean field $K$ is said to be \emph{spherically complete} if any decreasing sequence of closed balls has non-empty intersection.
\end{definition}

A valued field extension $L/K$ is said to be \emph{immediate} if it has the same value group and the same residue field. A field is called \emph{maximally complete} if it does not admit any proper immediate extension. By a classical result of Kaplansky (see e.g.\ \cite{schilling_valuations}, II.6, Theorem 8), it turns out that this concept is equivalent to spherical completeness. We thus have the following:

\begin{theorem}\cite[Satz 24]{Krull}
    Every non-Archimedean field admits a (non-unique) immediate extension that is spherically complete.
\end{theorem}

The following observation will be central for our result:

\begin{proposition}\label{prop_sphericallycomplete}
Let $K$ be a non-Archimedean field. Then $K$ is spherically complete if and only if every seminorm on a finite-dimensional $K$-vector space is diagonalizable.
\end{proposition}

\begin{proof}
 By \cite[Lemma 1.12]{boucksom_spaces_of_norms} the non-Archimedean field $K$ is spherically complete if and only if all norms on a finite-dimensional vector space over $K$ are diagonalizable. Our statement follows from the fact that a seminorm is diagonalizable if and only if the induced norm on the quotient modulo the kernel is diagonalizable.
\end{proof}

\begin{example}\label{ex_spherically_complete}
\begin{enumerate}[(i)]
    \item Every trivially valued field is obviously maximally (hence spherically) complete.
    \item Every discretely valued complete field is spherically complete, as by \cite[Proposition 3.1]{RemyThuillierWernerII} every norm on a finite-dimensional vector space is diagonalizable.
    \item In particular, for any prime number $p$ the field of $p$-adic numbers $\Q_p$ is local and spherically complete.
    \item The algebraic closure $\Q_p^a$ of the $p$-adic numbers is not complete (see \cite[Section 3.1.4]{robert_padic}).
    \item The completion $\C_p$ of $\Q^a_p$ is still algebraically closed, but not spherically complete. 
     \item Over $\C$, the field of rational functions $\C(t)$ is not Cauchy complete, as the Cauchy sequence $\left(\sum_{i=0}^n \frac{1}{i!}t^i\right)_{n\in\mathbb{N}}$ has no limit. Its completion is the field of (formal) Laurent series $\C(\!(t)\!)$, which is spherically complete, because the valuation is discrete.
    \item For any field $k$ let $k\{\!\{t\}\!\} =\bigcup_{n \geq 1} k(\!(t^{1/n})\!)$ be the field of Puiseux series. It is not complete. Its completion is the Levi-Civita field \cite[Theorem 4.10]{barria_summary_non_Archimedean}, in which the exponents of a series need not have a common denominator, but for any upper bound, there are only finitely many exponents with non-zero coefficients. 
    The Levi-Civita field is not spherically complete, its spherical completion is the Malcev-Neumann field $k(\!(t^{\Q})\!)$ of power series with rational exponents and well-ordered support (see \cite[Example 1.1]{boucksom_spaces_of_norms}).
\end{enumerate}
\end{example}

Let $\calN(V)$ and $\calNbar(V)$ be the set of norms and respectively non-trivial seminorms on the dual vector space $V^{\ast}$ (note the dualization!).  
For $x\in \calNbar(V)$ we denote by $\vert\vert.\vert\vert_x$ the associated seminorm. We endow $\calNbar(V)$ with the coarsest topology that makes the natural evaluation maps
\begin{equation*}
\calNbar(V) \ni x\longmapsto \seminorm{v}_x\in\R
\end{equation*}
for all $v\in V^{\ast}$ continuous.
The space $\calN(V)$ has been first introduced by Goldman and Iwahori in \cite{GoldmanIwahori} over $K=\Q_p$.
\begin{remark}[Topology of $\calNbar(V)$]
\begin{enumerate}[(i)]
\item Equivalently, one may define the topology on $\calNbar(V)$ as the topology of pointwise convergence: a net $(x_\alpha)$ in $\calNbar(V)$ converges to a seminorm $x \in \calNbar(V)$ if and only if $\seminorm{v}_{x_\alpha}$ converges to $\seminorm{v}_x$ in $\R$ for all $v \in V^\ast$.
\item A basis of the topology on $\calNbar(V)$ is given by open subsets of the form
\begin{align*}
    U &= \big\{x \in \calNbar(V)\ \suchthat \seminorm{v_i}_x\in (a_i,b_i) \textrm{ for all } i=1,\dots,l \big\} \\
    &= ev_{v_1}^{-1}((a_1,b_1)) \cap \dots \cap ev_{v_l}^{-1}((a_l,b_l))
\end{align*}
for some $v_1,\ldots,v_l \in V^{\ast}$ and open intervals $(a_1,b_1), \ldots, (a_l,b_l) \subseteq \R$, where $ev$ denotes the natural evaluation map.

\end{enumerate}
\end{remark}
Two points $x,y\in\calNbar(V)$ are said to be \emph{homothetic}, written as $x\sim y$, if there is a constant $c>0$ such $\vert\vert.\vert\vert_x=c\cdot \vert\vert.\vert\vert_y$. Homothety defines an equivalence relation on $\calNbar(V)$ that restricts to an equivalence relation on $\calN(V)$. 

We denote by $\calN^{\diag}(V)\subseteq \calN(V)$ the subspace of diagonalizable norms and $\calNbar^{\diag}(V)\subset\calNbar(V)$ the subspace of diagonalizable seminorms. Note that seminorms that are homothetic to a diagonalizable one are themselves diagonalizable. Let $\calX(V)=\calN(V)/_\sim$.

\begin{definition}
The \emph{affine Bruhat--Tits building} of $\PGL(V)$ is defined to be the quotient space of diagonalizable norms by homothety:
\begin{equation*}
\calB(V)=\calN^{\diag}(V)/_\sim .
\end{equation*}
\end{definition}

 The quotient $\calBbar(V)=\calNbar^{\diag}(V)/_\sim \subseteq\calXbar(V)=\calNbar(V)/_\sim$ forms a natural bordification of $\calB(V)$ (in fact, by  Corollary \ref{cor_seminormcompact} below $\calXbar(V)$ is compact). In particular, for a spherically complete field $K$ the quotient $\calBbar(V) = \calXbar(V)$ is a natural compactification of $\calB(V)$. This expands on the construction in \cite{Werner_seminorms}.

When $V=K^{n+1}$, we write $\calX_n(K)$ and $\calXbar_n(K)$ for $\calX(K^{n+1})$ and $\calXbar(K^{n+1})$ respectively, similarly $\calB_n(K)$ and $\calBbar_n(K)$.

\begin{remark}(Topological structure)
\begin{enumerate}[(i)]
\item By \cite[Theorem 1.19]{boucksom_spaces_of_norms} the locus $\calN^{\diag}(V)$ of diagonalizable seminorms is dense in $\calN(V)$ with respect to the Goldman-Iwahori metric 
$$d\big(\seminorm{\cdot}_x,\seminorm{\cdot}_y\big)=\sup_{v \in V^{\ast}}\big(\log \seminorm{v}_x-\log \seminorm{v}_y\big).$$
The finiteness of the supremum follows from the fact that any two norms are equivalent (see \cite[Page 10]{boucksom_spaces_of_norms}). The  topology induced by this metric is finer than the one defined above since uniform convergence implies pointwise convergence. It follows that $\calB(V)$ is dense in $\calX(V)$. Hence also $\calBbar(V)$ is dense in $\calXbar(V)$. Indeed, a stratum of $\calBbar(V)$ of seminorms with a fixed kernel $W \subset V$ can be identified with $\calB(V/W)$. 

\item Despite the notation suggesting that $\calN(V)\subset \calNbar(V)$ is open, this need not be true in general, as we will show in Remark \ref{notopen}. On the other hand, it is when $K$ is local. Indeed, let $\{x_\alpha\}_{\alpha}$ be a net of seminorms, converging to $\bar x$, and let $\{v_\alpha\}_{ \alpha}$ be a net of unit vectors (with respect to some fixed norm on $V$) such that $\lvert\lvert v_\alpha \rvert\rvert_{x_\alpha}=0$. Since the field is local, the unit sphere is compact and so we can find a convergent subnet converging to $\bar v\neq 0$. We conclude that $\lvert\lvert \bar v\rvert\rvert_{\bar x}=0$ by continuity and thus  $\bar x$ is also a proper seminorm.
\end{enumerate}
\end{remark}

\begin{example}\label{ex_trivialvaluationdescription}
Let $K$ have trivial valuation. Then we can give an explicit description of the space of seminorms up to homothety on a vector space $V$ of dimension $r+1$. First recall that by Example \ref{ex_spherically_complete} Item~(i) the field $K$ is spherically complete and hence by Proposition \ref{prop_sphericallycomplete} we have $\calXbar(V)=\calBbar(V)$. We claim that there is a bijection:
$$\calBbar(V)\stackrel{1:1}{\longleftrightarrow} \big\{\big(0=V_0 \subsetneq  V_1 \subsetneq  \dots \subsetneq  V_l=V^{\ast},c_1>\dots>c_{l-1}\big)\ \big\vert \   c_1,\dots,c_{l-1} \in \Rbar_{> 0} \big\}_{l=1,\ldots,r+1},$$
where $0=V_0 \subsetneq  V_1 \subsetneq  \dots \subsetneq  V_l=V^{\ast}$ is a flag of subspaces of $V^{\ast}$. We allow the special case of $l=1$ where we just have the flag $(0\subsetneq V^{\ast})$ without any coordinates. 
\begin{proof}
Let $\seminorm{\cdot}\in \calNbar(V)$ be a representative of $x \in \calBbar(V)$, i.e.~a non-trivial seminorm on $V^\ast$. Then all balls around $0$ in $V^{\ast}$ are subspaces. By letting the radius vary, we obtain a unique filtration $0=V_0 \subsetneq  V_1 \subsetneq  \dots \subsetneq  V_l=V^{\ast}$ such that for all $i=1,\dots,l$ the restriction of $\seminorm{\cdot}$ to $V_i \setminus V_{i-1}$ is a constant function and the values are strictly increasing.

Let $d_i$ denote the constant value on $V_i \setminus V_{i-1}$ and set for $i=1,\dots,l-1$
$$c_i:= -\log \frac{d_i}{d_{l}}.$$
Note that any seminorm homothetic to $||.||$ yields the same filtration of subspaces and only multiplies all the $d_i$ simultaneously by a common scalar. This does not change the $c_i$ and thus we obtain a well-defined map.

Vice versa, every flag of subspaces $0=V_0 \subsetneq  V_1 \subsetneq  \dots \subsetneq  V_l=V^{\ast}$ together with coordinates $c_1>\dots>c_{l-1}\in \Rbar_{> 0}$ gives rise to coordinates $d_i$ via the formula above by setting $d_{l}=1$, and thus we obtain a well-defined diagonalizable seminorm with 
$$||.||\big\vert_{V_i \setminus V_{i-1}}=d_i$$
and generic value $1$. 
By construction, these maps are inverses of each other.
\end{proof}
\end{example}

\begin{remark}
The bijection in Example \ref{ex_trivialvaluationdescription}  restricts to a bijection
$$\calB(V)\stackrel{1:1}{\longleftrightarrow} \big\{\big(0=V_0 \subsetneq  V_1 \subsetneq  \dots \subsetneq  V_l=V^{\ast},c_1>\dots>c_{l-1}\big)\ \big\vert\   c_1,\dots,c_l \in \R_{> 0} \big\}.$$
This can be identified with the cone over the order complex of the lattice of nontrivial subspaces of~$V^{\ast}$.
The latter comes with a natural weak topology, where a set is open if and only if its intersection with each cone is relatively open in that cone. This turns this set into the colimit of all cones corresponding to a fixed filtration.
However, the topology on the building is much coarser than the weak topology of the cone complex, as the following example will show.
\end{remark}

\begin{example}\label{B_1(K)trivial}
We consider the case where $K$ is any infinite field with trivial valuation. Then we can identify $\calBbar_1(K)$ with the set:
$$\big\{(0\subsetneq V_1\subsetneq (K^2)^{\ast},c) \ \big\vert\ c \in \Rbar_{\geq 0} \big\}\cup \big\{(0\subsetneq (K^2)^{\ast})\big\}.$$
Here $\eta:=(0\subsetneq (K^2)^{\ast})$ corresponds to the homothety class of a norm that is constant away from $0$.
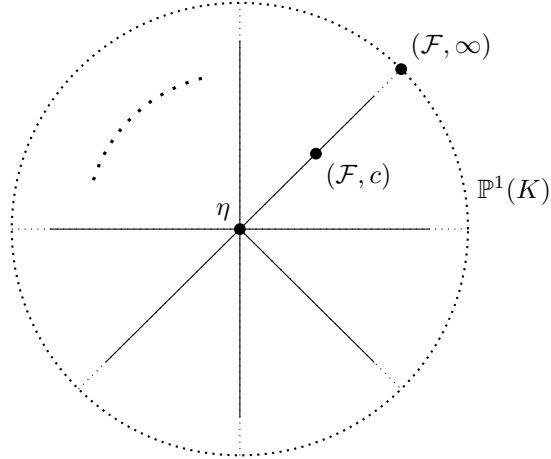
\begin{figure}[h]
\begin{tikzpicture}

\draw[black, dotted, thick](0,0) circle (3) {};

\draw[] (-2.5,0) -- (2.5,0);
\draw[dotted] (-3,0) -- (3,0);
\draw[] (0,-2.5) -- (0,2.5);
\draw[dotted] (0,-3) -- (0,3);
\draw[] ({-sqrt(6.25/2)},{-sqrt(6.25/2)}) -- ({sqrt(6.25/2)},{sqrt(6.25/2)});
\draw[dotted] ({-sqrt(9/2)},{-sqrt(9/2)}) -- ({sqrt(9/2)},{sqrt(9/2)});
\draw[] (0,0) -- ({sqrt(6.25/2)},{-sqrt(6.25/2)});
\draw[dotted] (0,0) -- ({sqrt(9/2)},{-sqrt(9/2)});

\draw[very thick, loosely dotted] (-0.5,2) arc (atan(1/4)+90:180-atan(1/4):2) ;

\draw[] (3,0.2) node[anchor=south west]{$\P^1(K)$};
\filldraw[black] (0,0) circle (2pt) node[anchor=south east]{$\eta$};
\filldraw[black] ({sqrt(9/2)},{sqrt(9/2)}) circle (2pt) node[anchor=south west]{$(\mathcal{F},\infty )$};
\filldraw[black] (1,1) circle (2pt) node[anchor=north west]{$(\mathcal{F},c )$};

\end{tikzpicture}
\caption{The building $\calBbar_1(K)$ of a trivially valued field, where $\mathcal{F}$ is shorthand for ``$0\subsetneq V_1\subsetneq (K^2)^{\ast}$''. The coordinate $c$ is a positive real number. A norm in the homothety class corresponding to $(\mathcal{F},c)$ has generic value $1$, and value $e^{-c}$ on $V_1\setminus\{0\}$. In the case of $c=\infty$, we have a proper seminorm with kernel $V_1$.}
\label{fig:example_trivial_valuation}
\end{figure}
Each cone corresponds to a one-dimensional subspace and the point at infinity of each cone corresponds to the homothety class of a proper seminorm. In rank one, the homothety class of a proper seminorm is uniquely determined by its kernel, and thus we can identify $\calBbar_1(K)\setminus \calB_1(K)$ with $\P^1(K)$.

A basis of the topology of $\calBbar_1(K)$ is given as follows: choose finitely many vectors $v_1,\dots,v_l\in (K^2)^{\ast}$ and intervals $(a_1,b_1),\dots,(a_l,b_l) \subset \R$ and define
\[U=\aset{x\in \calBbar_1(K)\ \suchthat \exists \textrm{ a representative } ||.||\in x \textrm{ with }||v_i||\in (a_i,b_i) \textrm{ for all } i=1,\dots,l }.\]
In particular, if such a set is a neighbourhood of $\eta$, then it contains all cones corresponding to subspaces which do not contain any vector $v_1,\dots,v_l$, i.e.~all but finitely many. This is of course not necessarily true for open neighborhoods of $\eta$ in the weak topology of the cone complex.
\end{example}

The space $\calB(V)$ has the structure of an affine building in the sense of \cite[Definition 1.9]{RemyThuillierWerner_survey} (see e.g.~\cite[III.1.2]{Parreau_thesis} for a proof). We refer the reader also to \cite{BennettSchwer} for a discussion of the various axiom systems for affine buildings over possibly non-discretely valued fields. For us, however, the most important notion is the following:

\begin{definition}
An \emph{apartment} $\calA(\mathbf{e})$ in $\calB(V)$ (associated to a basis $\mathbf{e}$ of $V^{*}$) is given by the image in $\calB(V)$ of 
\begin{equation*}\begin{split}
\R^{n+1}&\hooklongrightarrow \calN^{\diag}(V) \\
\vec{a}&\longmapsto \vert\vert.\vert\vert_{\mathbf{e},\vec{a}}  \ .
\end{split}\end{equation*}   
\end{definition}
Every apartment is a closed subset homeomorphic to $\R^{n+1}/\R\simeq \R^{n}$. The above parametrization can be extended to $\Rbar^{n+1}-\big\{(\infty,\ldots,\infty)\big\}$. Its image, denoted by $\calAbar(e)$, is the closure of $\calA(e)$ and may be naturally identified with $\T\P^n=\Rbar^{n+1}-\{(\infty ,\dots,\infty)\}/\ \R \cdot \mathbbm{1}$. We refer to $\calAbar(e)$ as a \emph{compactified apartment}.

\begin{example}
In Example \ref{B_1(K)trivial}, every apartment is the union of exactly two one-dimensional cones: a (homothety class of a) seminorm is diagonalizable by a basis 
$\{v_1,v_2\}\subset (K^2)^{\ast}$ if and only if it lies in the cone corresponding to $\langle v_1 \rangle$ or $\langle v_2 \rangle$.  
\end{example}

Our conventions are chosen so that the associations $V\mapsto \calXbar(V)$ and $V\mapsto \calBbar(V)$ are covariant functors from finite-dimensional $K$-vector spaces to topological spaces. In particular, any embedding $\iota: V \hookrightarrow W$ induces an embedding $\calBbar(\iota): \calBbar(V)\hookrightarrow \calBbar(W)$.
We have a natural operation of $\PGL(V)$ on $\calXbar(V)$, which respects diagonalizability and non-degeneracy. It is easy to show, that for $\dim V>1$ this operation of $\PGL(V)$ on $\calB(V)$ is transitive if and only if the valuation on $K$ is surjective.

\begin{example}
In the case of a discretely valued field $K$, the affine building $\calB_r(K)$ is a flag simplicial complex whose vertices correspond to equivalence classes of lattices (see Section \ref{sec:lattices}). 
Let $\OO_K$ denote the valuation ring, $\pi$ a uniformiser, and $k=\OO_K/(\pi)$ the residue field. A \emph{lattice} in $K^{r+1}$ is a free $\OO_K$-submodule of rank $r+1$. Two lattices $L_1, L_2$ are \emph{homothetic} if $L_1 = c L_2$ for some $c \in K^\ast$. Two homothety equivalence classes $\Lambda_1, \Lambda_2$ of lattices are \emph{adjacent} if there are representatives $L_1$ and $L_2$ such that $\pi L_1 \subset L_2 \subset L_1$. The simplices of $\calB_r(K)$ correspond to flags of adjacent lattices $\pi L_1\subset L_k\subset\ldots\subset L_1$. A lattice $L$ corresponds to an integral norm $|.|_L$ given by 
\begin{equation*}|x|_L := \inf\big\{|\lambda|\ \big\vert\ \lambda \in K^\ast, \lambda^{-1} x \in L\big\}
\end{equation*}
for $x\in K^{r+1}$, and one recovers the lattice as the closed unit ball of the norm. Given a lattice $L$, the star of $[L]$ in $\calB_r(K)$ can be identified with the spherical building $\calB_r(k)$ (see Example \ref{ex_trivialvaluationdescription}) by sending a flag of lattices $\pi L\subset L_k\subset\ldots\subset L$ to its image in $L/\pi L$, which is a flag of $k$-linear subspaces.
\end{example}

\begin{example}
Let $r=1$ and $K=\Q_p$. Then the affine Bruhat-Tits building $\calB_1(\Q_p)$ is an infinite tree whose vertices have valency $p + 1$. Let $\mathbf{e} = (e_1,e_2)$ be a basis of $(\Q_p^2)^\ast$. The apartment $\mathcal{A}(\mathbf{e})$ is an infinite path in the tree which uses all $\Z_p$-lattices with basis $(p^{u_1}e_1,p^{u_2}e_2)$ where $(u_1,u_2) \in \Z^2$. See Figure \ref{fig:examplePGL_2} for the compactified Bruhat-Tits tree of $\Q_2$. The open part is the usual trivalent infinite tree $\calB_1(\Q_2)$. The boundary $\calBbar_1(\Q_2) \setminus \calB_1(\Q_2)$ (illustrated by the circle) can be identified with $\P^1(\Q_2)$ since any homothety class of a non-trivial proper seminorm on $(\Q_2^2)^\ast$ can be identified with its kernel, which is a $1$-dimensional subspace of $(\Q_2^2)^{\ast}$.
    \begin{figure}
        
    \begin{tikzpicture}[
  grow cyclic,
  level distance=1cm,
  level/.style={
    level distance/.expanded=\ifnum#1>1 \tikzleveldistance/1.5\else\tikzleveldistance\fi},
  level 1/.style={sibling angle=120},
  level 2/.style={sibling angle=90},
  level 3/.style={sibling angle=90},
  level 4/.style={sibling angle=45},
  nodes={circle, fill = black,draw,inner sep=+0pt, minimum size=2pt},
  ]
\path[rotate=90]
  node (a) {}
  child foreach \cntI in {1,...,3} {
    node {}
    child foreach \cntII in {1,...,2} { 
      node {}
      child foreach \cntIII in {1,...,2} {
        node {}
        child foreach \cntIV in {1,...,2} {
          node {}
          child foreach \cntV in {1,...,2} {}
        }
      }
    }
  };
\draw[] (2.8,1) node[fill = none, draw = none, anchor=south west]{$\P^1(\Q_2)$};
\draw[black, dotted, thick](0,0) circle (2.8) {};

\end{tikzpicture}
\caption{The affine Bruhat-Tits building $\calBbar_1(\Q_2)$}
        \label{fig:examplePGL_2}
    \end{figure}
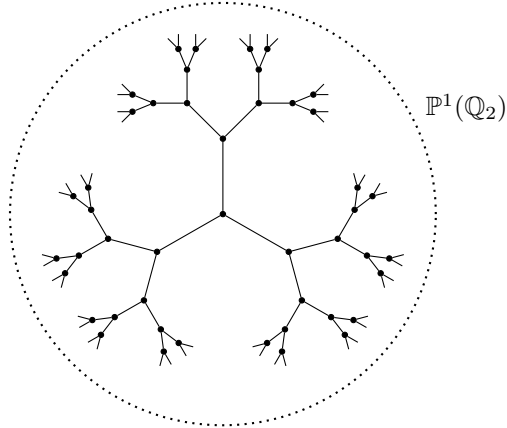
\end{example}


\section{Analytification and tropicalization}\label{section_analytification&tropicalization}

For a complete non-Archimedean field $K$ we recall the Berkovich analytification for any locally finite type scheme over $K$. We describe the relations between the Berkovich space, the space of seminorms on $(K^{r+1})^\ast$, and the compactified building. This allows us to define a tropicalization map from the space of seminorms on $(K^{r+1})^\ast$ to  tropical projective space.

Let $A$ be a finitely generated $K$-algebra and write $U=\Spec A$. 
\begin{definition}
A \emph{(multiplicative) seminorm} on $A$ is a map $\vert.\vert\colon A\rightarrow \R$ such that 
\begin{enumerate}[(i)]
\item $\vert f\vert \geq 0$ for all $f\in A$ as well as $\vert a\vert =\vert a\vert_K$ for all $a\in K$;
\item $\vert f\cdot g\vert =\vert f\vert \cdot \vert g\vert$ for all $f,g\in A$; and 
\item $\vert f + g \vert \leq \max\big\{\vert f\vert , \vert g\vert\big\}$ for all $f,g\in A$. 
\end{enumerate}    
\end{definition}
We think of the set $U^{\an}$ of multiplicative seminorms on $A$ as a space, the \emph{analytification} of $U$ in the sense of Berkovich \cite{Berkovich_book}; we write $x\in U^{\an}$ for a point in $U^{\an}$ as well as $\vert.\vert_x$ for the associated seminorm. The (analytic) topology of $U^{\an}$ is the coarsest that makes all evaluation maps 
\begin{equation*}\begin{split}
U^{\an}&\longrightarrow \R\\
x &\longmapsto \vert f\vert_x
\end{split}\end{equation*}
for $f\in A$ continuous. For a scheme $X$ that is locally of finite type over $K$, we define its \emph{analytification} $X^{\an}$ locally as above, and globally by gluing over affine open covers. See \cite[Chapter 3]{Berkovich_book} for details. 

\begin{remark}\label{kpointsdense}
    For an affine scheme $U$ of finite type over $K$, we have a natural inclusion $U(K)\to U^{\an}$ via $x\mapsto \left[f \mapsto |f(x)|\right]$. For any scheme $X$, locally of finite type over $K$, these inclusions on affine open subsets glue to an inclusion $X(K)\to X^{\an}$.  If the valuation on $K$ is non-trivial, the image of $X(\overline{K})$ is dense in $X^{\an}$ for an algebraic closure $\overline{K}$ of $K$.

    The association $X \mapsto X^{\an}$ is a covariant functor that commutes with the inclusion of the $K$-points of a scheme into its analytification. See \cite[Section 2.6]{Gubler_guide} for details.
\end{remark}

\begin{remark}
 To distinguish between multiplicative seminorms on a $K$-algebra and seminorms on a finite-dimensional $K$-vector space, we denote the former by $\vert.\vert_x$ and the latter by $\vert\vert.\vert\vert_x$.
 \end{remark}

Let $V$ be a vector space over $K$ of dimension $r+1$.
\begin{remark}[Analytification of projective spaces]
As explained in \cite[Section 2.1.1 ]{RemyThuillierWerner_survey}, the analytification of the projective space $\P(V)$ can be identified with the quotient of $\mathbb{A}(V)^{\an}-\{0\}$ modulo homothety.

Let $S^{\bullet}V^\ast$ be the symmetric algebra of the dual vector space $V^\ast$. This is a finitely generated graded $K$-algebra. Every choice of a basis $(e_0,\ldots,e_r)$ of $V^\ast$ induces an isomorphism of $S^{\bullet}V^\ast$ with the polynomial ring over $K$ in $r+1$ indeterminates. The affine space $\mathbb{A}(V)$ and projective space $\P(V)$ are defined as
\begin{equation*}
    \mathbb{A}(V) = \Spec (S^{\bullet}V^\ast) \qquad \text{ and } \qquad \P(V) = \Proj (S^{\bullet}V^\ast).
\end{equation*}
As said above, $\mathbb{A}(V)^{\an}$ consists of all multiplicative seminorms on $S^{\bullet}V^\ast$. Then the analytification $\P(V)^{\an}$ is the quotient of $\mathbb{A}(V)^{\an}-\{0\}$ by homothety: we define $x \sim y$ if and only if there exists a constant $c > 0$ such that for all $f \in S^{n}V^{\ast}
$ we have $\vert f \vert_x = c^{n} \vert f \vert_y $. 
Here the analytic topology on $\P(V)^{\an}$ is equal to the quotient topology.
\end{remark}

A multiplicative seminorm on $S^{\bullet}V^\ast$  induces a seminorm on $V^\ast = S^{1}V^\ast$ by restriction, hence we have a natural continuous map $\tau: \mathbb{A}(V)^{\an} \longrightarrow
\calNbar(V)$ such that $\tau(x) = 0$ if and only if $x=0$. Since this map is compatible with the equivalence relations, it descends to a continuous \emph{restriction map} 
\begin{equation*}\tau: \P(V)^{\an} \longrightarrow \calXbar(V)\ .
\end{equation*}

\begin{proposition}
Let $K$ be spherically complete. Then the  restriction map admits a section $$J:\calXbar(V)=\calBbar(V)\to \P(V)^{\an}.$$
Given a diagonalizable seminorm $\vert\vert.\vert\vert$, we may choose a basis $e_0,\dots, e_r$ of $V^{\ast}$ and $c_0,\dots, c_r \in \R_{\geq 0}$ such that 
$$\Big|\Big|\sum \lambda_i e_i\Big|\Big| = \max_i\{c_i|\lambda_i|\}.$$
The multiplicative seminorm $J(||.||)$ on $S^{\bullet}V^{\ast}$ is defined by
$$J\big(||.||\big)\bigg( \sum_{I=(i_0,\dots,i_r)} a_I e_0^{i_0}\dots e_r^{i_r} \bigg) = \max_I\big\{|a_I|c_0^{i_0}\ldots c_r^{i_r} \big\}.$$
When $K$ is a local field, the section $J$ is continuous. 
\end{proposition}

\begin{proof}
 We refer the reader to \cite[Section 3]{RemyThuillierWernerII} for details on this construction. Note that in \cite{RemyThuillierWernerII} the authors  assume that $K$ is local, although everything but the continuity of $J$ goes through when $K$ is spherically complete. 
\end{proof}

\begin{remark}
We do not know whether $J$ is continuous, when $K$ is not local. Luckily this statement is not needed in the remainder of this article.
\end{remark}

\begin{proposition}\label{restriction_surjective}
Let $K$ be any complete non-Archimedean field. Then the restriction map 
$$\tau: \P(V)^{\an}\to \calXbar(V)$$ is surjective. 
\begin{proof}
 We want to construct a section $\calXbar(V)\to \P(V)^{\an}$ in the general case. We pick any spherically complete extension $L/K$
and denote $V\otimes_K L$ by $V_L$. Then for any seminorm $||.||$ on $V^{\ast}$ we have an induced seminorm $||.||_L$ on $V^{\ast}_L$ given by
$$||w||_L= \inf \max_i\big\{|\lambda_i| \cdot||v_i||\big\}\ , $$
where the infimum runs over all possible decompositions $w=\sum_i \lambda_i v_i$ with $\lambda_i \in L$ and $v_i \in V^{\ast}$. It has been shown in \cite[Proposition 1.25]{boucksom_spaces_of_norms}  that the map $\calXbar(V)\to \calXbar(V_L)$ given by $||.||\mapsto ||.||_L$ is injective and that this map is a section of the natural restriction map $\calXbar(V_L)\to \calXbar(V)$.
Now, we define a section of $\tau$ by composing
$$\calXbar(V)\longrightarrow \calXbar(V_L)=\calBbar(V_L)\xlongrightarrow{J} \P(V_L)^{\an} \longrightarrow \P(V)^{\an}$$
where the last map is the restriction of a multiplicative seminorm on the symmetric algebra $S^{\bullet}V^{\ast}\subseteq S^{\bullet}V_L^{\ast} $.
\end{proof}
\end{proposition}

\begin{corollary}\label{cor_seminormcompact}
The space $\calXbar(V)$ is compact.
\end{corollary}
\begin{proof}
    This follows immediately from surjectivity of $\tau: \P(V)^{\an}\to \calXbar(V)$ and compactness of $\P(V)^{\an}$ (see e.g.~ \cite[Theorem 3.4.8 and 3.5.3]{Berkovich_book}).
\end{proof}

\begin{remark}
 We do not know if the section constructed in the proof of Proposition \ref{restriction_surjective} is independent of the choice of the spherical completion (which need not be unique in some cases in positive characteristic, see  \cite[Theorem 6.17]{barria_summary_non_Archimedean}), or if the section is continuous.
\end{remark}

\begin{remark}\label{notopen}
    We can now show that $\calN(V)\subseteq \calNbar(V)$ need not be open. 
    Let $K$ be algebraically closed. The set of $K$-points of $\P(V)$ is dense in $\P(V)^{\an}$ and gets mapped to homothety classes of proper diagonalizable seminorms in $\calXbar(V)$, thus its image is also dense.
    Consequently, neither inclusion $\calX(V)\subseteq\calXbar(V)$, nor $\calB(V)\subseteq\calBbar(V)$, nor $\calN(V)\subseteq \calNbar(V)$ can be open in this case.
\end{remark}

From now on we consider the vector space $V=K^{n+1}$ together with its standard basis $\mathbf{e} = (e_0,\ldots,e_n)$ and the associated dual basis $\mathbf{e}^\ast = (e_0^\ast,\ldots, e_n^\ast)$ of $V^\ast$. This identifies $\A(V)$ and $\P(V)$ with $\A^{n+1}=\Spec K[t_0,\ldots, t_n]$ and $\P^n=\Proj K[t_0,\ldots, t_n]$ respectively. As explained e.g.~in \cite{Payne_anallimittrop}, there is a natural continuous tropicalization map 
\begin{equation*}\begin{split}
    \trop_{\A^{n+1}}\colon (\A^{n+1})^{\an}&\longrightarrow \Rbar^{n+1}\\
    x&\longmapsto \big(-\log\vert t_0\vert_x, \ldots, -\log\vert t_n\vert_x\big) 
   \end{split}
\end{equation*}
that is compatible with the diagonal $\G_m$-operation. Therefore, this induces a tropicalization map 
\begin{equation*}\begin{split}
    \trop_{\P^n}\colon(\P^n)^{\an}&\longrightarrow \T\P^n\\
    x&\longmapsto \big[-\log\vert t_0\vert_x:\cdots: -\log\vert t_n\vert_x\big] \ .
\end{split}\end{equation*}
Tropicalizing the analytification of the big torus in $\P^n$ yields only elements with finite coordinates:
$$\trop_{\P^n}(\G_m^n)^{\an} = \R^{n+1}/\mathbbm{1}.$$

\begin{proposition}\label{tropicalization_factors}
The tropicalization map $\trop_{\P^n}$ factors as 
\begin{equation}\label{eq_factorization}
    (\P^n)^{\an}\xlongrightarrow{\tau}\calXbar_n(K)\xlongrightarrow{\trop_{\calXbar_n}} \T\P^n \ ,
\end{equation}
where $\trop_{\calXbar_n}\colon \calXbar_n(K)\rightarrow\T\P^n$ is a continuous and surjective map given by associating to a seminorm $\vert|.|\vert_x\colon V^\ast\rightarrow \R$ the tuple 
\begin{equation*}
    \trop_{\calXbar_n}(x)=\big[-\log\vert| e_0^\ast|\vert_x :\cdots : -\log\vert| e_n^\ast|\vert_x\big] \in \T\P^{n} \ .
\end{equation*}
\end{proposition}

\begin{proof}
 
The tropicalization map $\trop_{\calXbar_n}$ is well-defined, since, for two seminorms $\vert\vert.\vert\vert_x,\vert\vert.\vert\vert_y$ on $V^\ast$ together with a homothety $x\sim y$ we have a $c>0$ such that $\vert\vert.\vert\vert_x=c\cdot \vert\vert.\vert\vert_y$ and thus 
\begin{equation*}\begin{split}
    \trop_{\calXbar_n}(x)&=\big[-\log\vert\vert e_0^\ast\vert\vert_x :\cdots : -\log\vert\vert e_n^\ast\vert\vert_x\big]  \\
    &= \big[-\log c-\log\vert\vert e_0^\ast\vert\vert_y :\cdots : -\log c -\log\vert\vert e_n^\ast\vert\vert_y\big]\\
    &= \big[-\log\vert\vert e_0^\ast\vert\vert_y :\cdots : -\log\vert\vert e_n^\ast\vert\vert_y\big] \\&=\trop_{\calXbar_n}(y) \ .
\end{split}\end{equation*}
It is continuous, since it is given by evaluation maps in each coordinate, and surjective, since the compactified apartment map
\begin{equation*}\begin{split}
    \Rbar^{n+1}-\big\{\infty^{n+1}\big\}&\longrightarrow \calBbar_n(K)\\
    (a_0,\ldots, a_{n})&\longmapsto \vert\vert.\vert\vert_{\mathbf{e}^\ast,(a_0,\ldots, a_{n})}
\end{split}\end{equation*}
induces a continuous section $\T\P^n\rightarrow\calBbar_n(K)\subseteq\calXbar_n(K)$ of $\trop_{\calXbar_n}$. The factorization \eqref{eq_factorization} follows from the observation that, under the identification $S^{\bullet}V^\ast\simeq K[t_0,\ldots, t_n]$ the linear one-form $e_i^\ast$ is naturally identified with the linear polynomial $t_i$. Therefore, we have $\vert t_i\vert_x=\vert\vert e_i^\ast\vert\vert_{\tau(x)}$ for all $x\in (\P^n)^{\an}$ and $i=0,\ldots, n$. This implies
\begin{equation*}
    \begin{split}
        \trop_{\calXbar_n}(\tau(x))&=\big[-\log\vert\vert e_0^\ast\vert\vert_{\tau(x)} :\cdots : -\log\vert\vert e_n^\ast\vert\vert_{\tau(x)}\big]  \\
        &=\big[-\log\vert t_0\vert_{x} :\cdots : -\log\vert t_n\vert_{x}\big] \\&= \trop_{\P^n}(x) 
    \end{split}
\end{equation*}
for all $x\in(\P^n)^{\an}$.
\end{proof}

Similarly, tropicalizing the non-compactified space of norms, one obtains all finite points:
$$\trop_{\calXbar_n}(\calX_n(K))=\R^{n+1}/\mathbbm{1}.$$

Let $Y\subseteq \P^n$ be a Zariski-closed subscheme. Then, following \cite{Payne_anallimittrop, Gubler_guide}, the \emph{tropicalization} of $Y$ is defined to be the projection 
\begin{equation*}
    \Trop_{\P^n}(Y)=\trop_{\mathbb{P}^n}(Y^{\an})
\end{equation*}
of $Y^{\an}$ under $\trop_{\P^n}$ into $\T\P^n$. Let $L$ be an algebraically closed extension of $K$ with non-trivial absolute value $|.|_L$. By \cite[Proposition 3.8]{Gubler_guide} the tropicalization  $\Trop_{\P^n}(Y)$ is equal to the closure of 
    $$\Big\{\big[-\log(|t_0|_L):\cdots:-\log(|t_n|_L)\big] \ \Big\vert\  [t_0:\cdots:t_n]\in Y(L)\subset \P^n(L)\Big\} \subseteq \T\P^n.$$

\begin{proposition}\label{prop:tropIsProj}
For a linear embedding $\iota\colon\P^r\hookrightarrow\P^n$ the tropicalization $\Trop(\P^r,\iota)=\trop_{\P^n}\big(\iota(\P^r)^{\an}\big)$ is equal to the projection of $\calXbar(\iota)\big(\calXbar_r(K)\big)\subseteq \calXbar_n(K)$ under $\trop_{\calXbar_n}$ into $\T\P^n$.
\end{proposition}
\begin{proof}
 Since the restriction map $\tau$ is surjective, the commutativity of 
 \begin{equation*}\begin{tikzcd}
 (\P^{r})^{\an}\arrow[d,"\iota^{\an}"']\arrow[r,"\tau"]&\calXbar_r(K)\arrow[d,"\calXbar(\iota)"]\\
 (\P^n)^{\an}\arrow[r,"\tau"]&\calXbar_n(K) 
 \end{tikzcd}
\end{equation*}
 implies that $\calXbar(\iota)\big(\calXbar_r(K)\big)=\tau\big(\iota^{\an}\big(\P^r)^{\an}\big)$. The factorization of the tropicalization map in Propositon \ref{tropicalization_factors} then yields the claim.
\end{proof}

Henceforth, we will define for any linear embedding $\iota:\P^r\hookrightarrow \P^n$ the composition $\pi_{\iota}$ by
$$\pi_{\iota}:=\ \trop_{\calXbar_n}\circ\ \calXbar(\iota): \calXbar_r(K)\to \Trop(\P^r,\iota).$$ 
Note that $\pi_\iota$ is continuous and surjective. A direct computation shows that if $\iota=\big[f_0:\ldots:f_n\big]$ for $f_0,\dots,f_n \in (K^{r+1})^{\ast}$ we have
\begin{equation*}\pi_{\iota}\big(x\big) = \big[-\log(||f_0||_x):\cdots : -\log(||f_n||_x)\big]
\end{equation*}
for all $x\in \calXbar_r(K)$.

\section{Limits of linear tropicalizations} 
In this section we will set up 
and prove Theorem \ref{mainthm_limits}. We first set up a category of linear embeddings such that tropicalization yields a covariant functor into the category of topological spaces.

\begin{definition}\label{def:category}
Let $I$ be the category of linear embeddings $\P^r \hookrightarrow U \subseteq \P^n$, where $U$ is a torus-invariant open subset of $\P^n$, with morphisms given by commutative triangles 
\begin{center}
        \begin{tikzcd}
        \P^r \ar[r,"\iota"] \ar[dr,"\iota'"'] & U \ar[d] \\
        & U'
        \end{tikzcd}
    \end{center}
where $U\to U'$ is a toric morphism.
\end{definition}

\begin{remark}
The codimension of the complement of $U\subseteq \P^n$ in Definition \ref{def:category} must be at least $r+1$. Thus, for all $j=0,\dots, r$, the Chow group of $U$ of codimension $j$ is isomorphic to that of $\P^n$. The degree of $\iota'$ being $1$ can be measured by intersecting its image with a generic  linear subspace of codimension~$r$. Since the same is true for $\iota$, we see that hyperplanes must pull back to hyperplanes along $U\to U'$, i.e.~this morphism is forced to be linear. \end{remark}

\begin{remark}

Allowing the toric morphism to be defined on a smaller torus-invariant open subset instead of the whole projective space endows our index category with many more morphisms, most notably coordinate projections. This makes $I$ into a cofiltered category: indeed, for any two objects $\iota_i\colon\P^r\hookrightarrow U_i\subseteq\P^{n_i},\ i=1,2$, we can find a third one dominating both, namely $\iota\colon\P^r\hookrightarrow U\subseteq\P^{N}$, where $N=n_1+n_2+1$ and \[U=\left(\P^N\setminus V(z_0,\ldots,z_{n_1})\cap\pr_1^{-1}(U_1)\right)\cap \left(\P^N\setminus V(z_{n_1+1},\ldots,z_{n_1+n_2+1})\cap\pr_2^{-1}(U_2)\right).\]
And for any two morphisms $f,g\colon U_1\to U_2$ commuting with $\iota_i$ as above, we can equalize them by defining $U_0=\{f=g\}\subseteq U_1$, and noticing that $\iota_1$ factors through $U_0$, and the closure of $U_0$ in $\P^{n_1}$ is a linear subspace $\P^{n_0}\subseteq\P^{n_1}$.
\end{remark}

The tropicalization $U^{\trop}$ of a torus-invariant open subset or $\P^n$ is a special case of the tropicalization of toric varieties, as introduced in  \cite[Section 3]{Payne_anallimittrop}. In our case this means removing those tropical torus-orbits from $\T\P^n$, for which the corresponding algebraic torus orbit is not contained in $U$. For a toric morphism $\varphi\colon U\rightarrow U'$ such that $\varphi\circ \iota=\iota'$, we have a natural induced map $\varphi^{\trop}\colon U^{\trop}\rightarrow U'^{\trop}$ such that $\varphi^{\trop}\big(\Trop(\P^r,\iota)\big)\subseteq \Trop(\P^r,\iota')$. In particular, if $\phi$ is a coordinate projection, then $\phi^{\trop}$ is the analogous tropical coordinate projection. We refer the reader to \cite{Payne_anallimittrop}, to  \cite[Section 3 and 5]{Rabinoff}, and to \cite[Section 6.2]{MaclaganSturmfels} for more details on the tropical geometry of toric varieties.

Let $\iota: \P^r \hookrightarrow U$ be a linear closed immersion, where $U$ is a torus-invariant open subset of $\P^n$, and let $\jmath:U\to \P^n$ be the inclusion. Then we have a homeomorphism of tropicalizations
$$\Trop(\P^r,\iota)\xrightarrow{\jmath^{\trop}} \Trop(\P^r, \jmath\circ \iota) \subseteq \T\P^n.$$
Using this observation we will henceforth identify any tropicalization arising from a linear morphism $\P^r \to U$ with a subset of $\T\P^n$.

\begin{lemma}
    Let $\iota:\P^r\hookrightarrow U\subseteq \P^n$ and $ \iota':\P^r\hookrightarrow U'\subseteq \P^{n'}$ be linear embeddings and $\varphi: U\to U'$ be a toric morphism with $\varphi\circ \iota = \iota'$. Then the diagram 
    \begin{center}
        \begin{tikzcd}
        \calXbar_r(K) \ar[r, "\pi_{\iota}"] \ar[dr, "\pi_{\iota'}"'] & \Trop(\P^r,\iota) \ar[d, "\varphi^{\trop}"] \\
        & \Trop(\P^r,\iota')
        \end{tikzcd}
    \end{center}
    commutes.
    \begin{proof}
        This follows immediately from Proposition \ref{tropicalization_factors} and the fact that the same holds for the Berkovich analytification \cite{Payne_anallimittrop}.
    \end{proof}
\end{lemma}

This lemma yields a well-defined continuous map
$$\calXbar_r(K)\xlongrightarrow{\varprojlim_{\iota \in I} \pi_{\iota}}\varprojlim_{\iota \in I} \Trop\big(\PP^r,\iota\big)$$ 
where the limit runs over all linear embeddings $\iota:\P^r \to U\subseteq \P^n$ into torus-invariant open subsets. The limit is endowed with the coarsest topology making all projections continuous; in particular, the limit topology is generated by (the preimage under projection of) all opens in all (finite) tropicalized linear spaces. The right-hand side is thus a pro-object in the category of topological spaces.

    \begin{theorem}[Theorem \ref{mainthm_limits}]\label{thm_limits}
    The map 
    $$\varprojlim_{\iota \in I} \pi_\iota\colon \overline{\calX}_r(K)\longrightarrow\varprojlim_{\iota \in I} \Trop\big(\PP^r,\iota\big)$$ 
    is a homeomorphism.
    \end{theorem}
    
    \begin{proof}        As $\varprojlim_{\iota \in I} \Trop\big(\PP^r,\iota\big)$ is a Hausdorff space and $\calXbar_r(K)$ is compact by Proposition \ref{restriction_surjective}, it suffices to show bijectivity.
        
        We first show the \textbf{injectivity} of $\varprojlim_{\iota \in I} \pi_\iota$. Assume that $\varprojlim_{\iota \in I} \pi_{\iota}(x)=\varprojlim_{\iota \in I} \pi_{\iota}(x')$ for $x,x' \in \calXbar_r(K)$. We choose representatives $y$ and $y'$ of the homothety classes $x$ and $x'$ respectively. 
        
        \emph{Claim:} The seminorms $y$ and $y'$ have the same kernel. \\
        By symmetry, it suffices to show one inclusion. Let $0 \neq f\in (K^{r+1})^*$ with $||f||_y=0$. Then by extending $f$ to a generating set $f,f_1,\dots,f_n$ of $(K^{r+1})^{\ast}$ and looking at the corresponding embedding $\iota=[f:f_1:\ldots:f_n]: \P^r \to \P^n$, we obtain that the first coordinate of $\pi_{\iota}(x)=\pi_{\iota}(x')$ equals $-\log ||f||_y=\infty$. By assumption, also $-\log ||f||_{y'}=\infty$ and thus also $||f||_{y'}=0$.
        
        Let now $f_0,f_1 \in (K^{r+1})^{\ast}$, with $||f_0||_y\neq 0$ and $||f_1||_y\neq 0$. Then also $||f_0||_{y'}\neq 0$ and $||f_1||_{y'}\neq 0$ and we need to show that $\frac{||f_0||_y}{||f_0||_{y'}}=\frac{||f_1||_y}{||f_1||_{y'}}$, as this immediately implies that $y,y'$ are homothetic and thus $x=x'$.
        We extend $f_0,f_1$ to a generating set $f_0,f_1,\dots, f_n$ of $(K^{r+1})^{\ast}$. Let $\iota=[f_0:\ldots:f_n]$ be the corresponding embedding from $\P^r$ into $\P^n$, then $\pi_{\iota}(y)=\pi_{\iota}(y')$ and thus their difference is a multiple of $\mathbbm{1}$, in particular 
        $$-\log \frac{||f_0||_y}{||f_0||_{y'}} = -\big(\log ||f_0||_y -\log ||f_0||_{y'}\big) = -\big(\log ||f_1||_y -\log ||f_1||_{y'}\big) = -\log \frac{||f_1||_y}{||f_1||_{y'}}.$$
        
        We now show the \textbf{surjectivity} of $\varprojlim_{\iota \in I} \pi_\iota$. Let $(y_{\jmath})_{\jmath \in I}\in \varprojlim_{\jmath\in I} \Trop\big(\PP^r,\jmath\big)$. First, we consider the identity $\id=\big[e_0^{\ast}:\ldots:e_r^{\ast}\big]:\P^r\to\P^r$. After a permutation of coordinates we may assume that the first coordinate $y_{\id,0}$ of $y_{\id}\in \T\P^r$ is not $\infty$. We will construct a seminorm $||.||$ with $||e_0^{\ast}||=1$ and $\pi_{\jmath}(||.||)=y_{\jmath}$ for all $\jmath\in I$.

        \emph{Claim:} For all linear embeddings $\iota=\big[e_0^{\ast}:f_1:\cdots:f_{n}\big]$, the first coordinate $y_{\iota,0}$ is not $\infty$.\\ 
        After composing with the corresponding embedding into projective space we can assume that the codomain of $\iota$ is $\P^n$.
        We consider the linear embedding $$\big[e_0^{\ast}:\cdots:e_r^{\ast}:f_1:\ldots:f_n\big] :\P^r \to U \subset \P^{r+n}$$
        where $U=\P^{r+n}\setminus V(x_0,\dots,x_r)\cup V(x_{0},x_{r+1},\dots,x_{r+n})$. Then we have a  projection $U \to \P^r$ onto the first $r+1$ coordinates and a projection $U \to \P^{n}$ given by $ [x_0 : \ldots:x_{r+n}] \mapsto [x_0:x_{r+1}:\ldots:x_{r+n}]$. Since $(y_{\jmath})_{\jmath\in J}$ is an inverse system, this shows that the first coordinate of $y_{\iota}$ cannot be $\infty$.

        \emph{Construction of the seminorm.}\\
        Let $f \in (K^{r+1})^{\ast}$. We choose an embedding $\jmath=\big[e_0^{\ast}:f:f_2:\cdots:f_n\big]:\P^r \to \P^n$ and define 
        $$||f||:=\exp(y_{\jmath,0}-y_{\jmath,1})$$
        where we set $\exp(-\infty)=0$. Note that $y_{\jmath,0}\neq \infty$ and that this does not depend on the representative of $y_{\jmath}\in \T\P^n$.
        
        We show that this is independent of the choice of $\jmath$. Let $\jmath'=\big[e_0^{\ast}:f:f'_2:\cdots:f'_{n'}\big]$. Similarly to before, consider the embedding
        $$\big[e_0^{\ast}:f:f_2:\ldots:f_n:f'_2\cdots:f'_{n'}\big] :\P^r \to U \subseteq \P^{n+n'-1}$$
        where $U=\P^{n+n'-1}\setminus V(x_0,\dots,x_n)\cup V(x_0,x_1,x_{n+1},\dots,x_{n+n'-1})$ as before. By the same argument, applying the projections to $\P^n$ and $\P^{n'}$ shows that $y_{\jmath,0}-y_{\jmath,1}=y_{\jmath',0}-y_{\jmath',1}$. 
        
        By a permutation automorphism one can show that for any linear embedding $\iota =\big[e_0^{\ast}:\ldots:f:\ldots\big]$, where $f$ is in the $i$-th entry, we have $||f||:=\exp(y_{\jmath,0}-y_{\jmath,i})$.

        We check that the constructed map is indeed a seminorm. 
        For $f \in (K^{r+1})^{\ast}$ and $\lambda \in K$ consider any embedding $\jmath=\big[ e_0^{\ast}:f:\lambda f: \ldots \big]$. Then, by Proposition \ref{prop:tropIsProj}, for every $y_{\jmath}\in \Trop(\P^r,\jmath)$ there is a class of a seminorm $\big[||.||'\big]\in \calXbar_r(K)$ with 
        \begin{align*}
            y_{\jmath}&=\pi_{\jmath}\big(\big[||.||'\big]\big) \\
            &= \big(\trop_{\calXbar_n}\circ \calXbar(\jmath)\big)\big(\big[||.||'\big]\big)\\
            &=\big[-\log ||e_0^{\ast}||':-\log ||f||':-\log ||\lambda f||': \cdots \big]
        \end{align*}
        and thus $y_{\jmath,1}+\val(\lambda)=y_{\jmath,2}$, where $\val$ is the valuation on $K$. Therefore 
                $$||\lambda f||=\exp(y_{\jmath,0}-y_{\jmath,2})=\exp\big(y_{\jmath,0}-(y_{\jmath,1}+\val(\lambda))\big) = |\lambda|\cdot||f||.$$
        For $f,g \in (K^{r+1})^{\ast}$, the inequality $||f+g||\leq \max\{||f||,||g||\}$ follows similarly by considering an embedding containing $f,g$ and $f+g$.

        By construction, the seminorm $||.||$ is an inverse image of $(y_{\jmath})_{\jmath\in I}$: Let $\jmath=\big[f_1:\ldots:f_n\big]: \P^r \longrightarrow U \subseteq \P^{n-1}$ be a linear embedding into a torus-invariant open subset $U \subseteq \P^{n-1}$. Consider the embedding $\jmath'=\big[e_0^{\ast}:f_1:\ldots:f_n\big]: \P^r \longrightarrow V \subseteq \P^{n}$ where $V$ is the complement of the intersection of the last $n$ coordinate hyperplanes. Since the projection of $V$ onto the last $n$ coordinates is a toric morphism, we have an induced projection map $\Trop(\P^r,\jmath') \longrightarrow \Trop(\P^r,\jmath)$. The same permutation argument as before shows 
        \begin{align*}
            \big[-\log\big(||f_1||\big):\dots:-\log(||f_n||\big)\big] 
            &= \big[-\log(\exp(y_{\jmath',0}-y_{\jmath',1})):\dots:-\log(\exp(y_{\jmath',0}-y_{\jmath',n}))\big] \\
            &= \big[y_{\jmath',1}-y_{\jmath',0}:\dots:y_{\jmath',n}-y_{\jmath',0}\big]\\
            &= \big[y_{\jmath',1}:\dots:y_{\jmath',n}\big]\\
            &=\big[y_{\jmath,0}:\dots : y_{\jmath,n-1}\big]. \qedhere
        \end{align*}
    \end{proof}

\begin{remark}\label{rem_different_categories}
Instead of the category $I$ used above, several subcategories would yield the same limit:
\begin{enumerate}[(a)]
    \item The full subcategory $I'$ of $I$ of non-degenerate embeddings where no coordinate equals $0$. Then $I'$ is cofinal in $I$, and thus the respective limits of tropicalizations are naturally isomorphic. 
    \item The full subcategory of $I$ of non-degenerate embeddings which are different in every coordinate.
    \item The wide subcategory of $I$, where instead of all (linear) toric morphisms we only allow coordinate projections. Note that these morphisms are the only ones used in the proof of \ref{thm_limits}.
    
\end{enumerate}

Following the proof of Theorem \ref{thm_limits} one can similarly show that the restriction of the map $\varprojlim_{\iota \in I'} \pi_\iota$ to the non-compactified space $\calX_r(K)$ induces a homeomorphism
$$\calX_r(K)\xlongrightarrow{\sim}\varprojlim_{\iota \in I'} \Trop\big(\iota(\PP^r)\cap \G_m^{n_{\iota}}\big)=\varprojlim_{\iota \in I'} \left(\Trop\big(\PP^r,\iota\big)\cap \R^{n_{\iota}+1}/ \R \cdot \mathbbm{1}\right),$$
where $n_{\iota}$ is the dimension of projective space that is the codomain of $\iota$.
\end{remark}

\begin{remark}
In \cite{GunturkunKisisel} the authors  consider a projective limit of tropicalizations with respect to all linear re-embeddings of a fixed affine variety. They, in particular, show that this construction recovers the whole Berkovich analytification in the case of an affine smooth algebraic curve. Theorem \ref{mainthm_limits} may be thought of as a natural linear-algebraic incarnation of the authors' ideas.
\end{remark}


\section{Valuated matroids and tropical linear spaces}\label{section_valuatedmatroids}
In this section, we recall some of the basic definitions and results on valuated matroids in the sense of Dress and Wenzel \cite{DressWenzel_valuatedmatroids}, in particular how to associate a (projective) tropical linear space to a valuated matroid and to describe its local structure.

\subsection{Essentials of Matroids}
A \emph{matroid} $M$ of rank $r$ is given by an arbitrary set $E$, called the \emph{ground set}, and a set $\calI$ of subsets of $E$, called \emph{independent} sets, such that the following axioms are satisfied:
\begin{enumerate}
    \item[(I1)] The empty set is independent.
    \item[(I2)] Subsets of independent sets are independent.
    \item[(I3)] If $A,B\in \calI$ and $|A|>|B|$, then there is $a \in A$ such that $B\cup \{a\} \in \calI$.
    \item[(I4)] If $A$ is an inclusion-wise maximal independent set, then $|A| = r$. 
\end{enumerate}

Note that if the ground set $E$ is finite, then axiom (I4) follows from (I3), but if $E$ is infinite then (I4) is a necessary axiom.  
Clearly these axioms are modeled after linear independence of a set of vectors whose span is $r$-dimensional; e.g.~we get a matroid by considering the linearly independent subsets of a subset $E$ of a vector space $K^n$.
The \emph{rank} of $A \subset E$ is the cardinality of a maximal independent in $A$.
The family $\mathscr{B}(M)$ of inclusion-wise maximal sets of $\calI$ are called the \emph{bases} of $M$, and by axiom (I2) they determine $\calI$. 
A \emph{circuit} $C\subset E$ is a minimal dependent set, and a \emph{flat} is a subset $F\subset E$ such that $|C\setminus F| \neq 1$ for all circuits $C$. A set is a flat if and only if adding any other element to it increases its rank. A \emph{loop} is an element in $E$ that is contained in no basis; equivalently the singleton with only this element is dependent. 

\subsection{Valuated matroids}
We begin with the following definition due to Dress and Wenzel \cite{DressWenzel_valuatedmatroids}.

\begin{definition}
A \emph{valuated matroid} of rank $r$ on a ground set $E$ is a function
\begin{equation*}
v\colon {E\choose r} \longrightarrow \Rbar
\end{equation*}
that fulfils the following axioms:
\begin{enumerate}[(i)]
\item There exists a subset $A\in {E\choose r}$ with $v(A)\neq \infty$. 

\item For all $A,B\in {E\choose r}$ and $a\in A-B$ we have the \emph{valuated basis exchange property}
\begin{equation*}
v(A)+v(B)\geq \min_{b\in B}\big\{v( b \cup A \setminus a)+v( a \cup B \setminus b)\big\} \ .
\end{equation*}
\end{enumerate}
\end{definition}

The elements $A\in {E\choose r}$ with $v(A)<\infty$ form the set of bases of a matroid, called the \emph{underlying matroid} of $v$. We explicitly do not require the underlying set $E$ to be finite.

\begin{remark}
 Valuated matroids on $E$ of rank $r$ are parameterized by the \emph{Dressian}~$\operatorname{Dr}(E, r) \subset  \Rbar^{{E\choose r}}$.
 This is the tropical variety of points $v \in \Rbar^{{E\choose r}}$ such that for all $\tau \in{E\choose r+1}$ and all $\sigma \in {E\choose r-1}$ the expression 
 \begin{align}
 \label{eq:EqsCutOutDressian}
     \min_{j\in\tau}\big( v(\tau\setminus j) +  v(\sigma\cup j) \big)
 \end{align} 
 attains the minimum at least twice. 
 It is straightforward to verify that the basis exchange axiom of valuated matroids is equivalent to the minimum in Equation~\eqref{eq:EqsCutOutDressian} being attained at least twice for all $\tau \in {E \choose r+1}$ and all $\sigma \in {E\choose r-1}$.
If $v$ is a point in the interior $ \operatorname{Dr}(E, r)^\circ =  \operatorname{Dr}(E, r) \cap  \RR^{{E\choose r}}$, then the underlying matroid of $v$ is the \emph{uniform matroid on $E$ of rank $r$}, i.e.~the bases are all  subsets of cardinality $r$. 
 Valuated matroids with different underlying matroids lie at the boundary. 
 That is, if $M$ is a matroid on $E$ of rank $r$, we obtain the \emph{Dressian with underlying matroid $M$} by intersecting with hyperplanes at infinity:
 \[\operatorname{Dr}(M)= \operatorname{Dr}(E,r) \cap \bigcap_{\substack{\sigma \in {E \choose r} \\ \sigma \text{ not a base} }} \big\{ v \in \Rbar^{E\choose r} \ \big\vert\  v(\sigma) = \infty \big\}. \]

\end{remark}

\begin{example}[Realizable valuated 
matroids]\label{realizableValuatedMatroids}
Let $K$ be a field with a non-Archimedean valuation $\val: K \rightarrow \Rbar$.
\begin{enumerate}[(a)]
    \item Let $\{f_0,\ldots,f_n\}$ be a generating subset of $K^{r+1}$. The map
\begin{align*}v\colon{[n]\choose r+1}&\longrightarrow\Rbar\\
A = \{a_0,\ldots,a_r\}&\longmapsto \val\Big(\det\begin{bmatrix}f_{a_0}&\cdots&f_{a_r}\end{bmatrix}\Big)\end{align*}
defines a valuated matroid or rank $r+1$. This follows from the \emph{Grassmann-Plücker} identity:
    \begin{equation*}
        \det(v_0,\ldots,v_r) \cdot \det(w_0,\ldots,w_r) = \sum_{i=0}^n \det(v_0,\ldots,v_{i-1},w_0,v_{i+1},\ldots,v_r) \cdot \det(v_i,w_0,\ldots,w_r)
    \end{equation*}
for all $v_0,\ldots,v_r,w_0,\ldots,w_r \in K^{r+1}$. All valuated matroids of this form are called \emph{realizable}.
\item Extending on (a), the map
    \begin{equation*}
w_{\univ}\colon{K^{r+1} \choose r+1}\longrightarrow \Rbar
\end{equation*}
induced by the permutation-invariant map $\val\circ\det\colon K^{(r+1)\times (r+1)}\rightarrow \Rbar$ is a valuated matroid, called the \emph{universal realizable matroid}.
    \item If the valuation on $K$ is trivial, then in (a) we have 
    $$v(A)=\begin{cases}
    0 & \text{if $A$ is a basis,}\\
    \infty & \textrm{if $A$ is dependent}.
    \end{cases}$$
    Thus, the notion of an ordinary matroid that is realizable over a fix trivially valued field corresponds exactly to that of a valuated matroid that is realizable over the same trivially valued field.
\end{enumerate}
\end{example}

\subsection{Tropical linear spaces and matroid polytopes}
\label{subsection_matroid_polytopes}
Valuated matroids $v$ with ground set $E$ and rank~$r+1$ have an associated polyhedral complex $\calL(v)$ of pure dimension~$r$ in $ \TT\PP^{E}$ that is connected through codimension~1 and which satisfies intersection-theoretic properties analogous to linear spaces. 
Hence the $\calL(v)$ are called \emph{tropical linear spaces}.
We recall the definition, some properties, the associated Matroid polytope, and a regular subdivision induced by $v$ that is dual to $\calL(v)$.
When the matroid is realizable, this recovers the coordinate-wise tropicalization of any linear subspace associated to~$v$.

For the following let $E = \{0,\ldots,n\}$ be the (finite) ground set of a valuated matroid $v$ of rank $r+1$ with underlying matroid $M$. 
\begin{definition}\label{definition_tropicallinearspace}
The \emph{tropical linear space} $\calL(v)\subset\TT\PP^{n}$ associated to $v$ is the locus of those $(u_e)_{e \in E} \in \T\P^{n}$ such that for any  $\tau \in{E \choose r+2}$ the minimum in $v(\tau \setminus e\}) + u_e$ is attained at least twice.
\end{definition}
If some $f \in E$ is a loop, by taking $\tau=B\cup \{f\}$ for any basis $B$, one can see that $u_f=\infty$ for each $(u_e)_{e\in E} \in \calL(v)$. 
Thus, adding or deleting loops only yields homeomorphic associated tropical linear spaces. Consequently, for simplicity, we now assume that $M$ has no loops.

Now, we give a characterization of matroids in terms of polytopes. We use the following notation for the indicator vector in $\RR^E$ of a set $A \subset E$:
\[e_A=\sum_{i\in A}e_i\in \R^E.\] 
\begin{definition}    
The \emph{matroid polytope} $P_M$ of a matroid $M$ is the convex hull of 
\begin{equation*}\big\{e_B \ \big\vert\ B \in \mathscr{B}(M)\big\} \subseteq \R^{n+1} .
\end{equation*}
\end{definition}

The valuated matroid $v$ can be regarded as a height function on the vertices of the polytope $P_M$ giving rise to the lifted polytope 
$\Gamma(v)$, which is defined to be the convex hull of 
\begin{equation*}\big\{(e_{B}, v(B)) \in \R^{n+2}\mid B \in \mathscr{B}(M)\big\} \ .
\end{equation*}
Projecting the lower facets of $\Gamma(v)$ back to $\R^{n+1}$ induces a polytopal subdivision $\mathcal{D}_v$ of $P_M$, called the \emph{regular subdivision} induced by $v$. By \cite[Proposition 2.2]{YuYuster2006} a real-valued function from the vertices of $P_M$ is a 
valuated matroid if and only if all the faces of the induced regular subdivision are matroid polytopes.

A vector $u \in \R^{n+1}$ selects a face of the regular subdivision induced by $v$ by taking the convex hull of all vertices $e_B$ of $P_M$ such that $v(B) - u \cdot e_B$ is minimized. 
Such a face corresponds to the polytope of a matroid, the so-called \emph{inital matroid} $M_u$ of $M$ at $u$. 
For a loopless valuated matroid $v$ we have that: 

\begin{proposition}[{\cite{FinkOlarte} or \cite[Proposition 2.3]{Speyer}}]
    The interior of the tropical linear space $\calL(v)$, i.e.~the points with finite coordinates and satysfying Equation~\eqref{eq:EqsCutOutDressian}, equals the set 
    \[\calL(v)^{\circ} = \{u \in \R^{n+1} \mid M_u \text{ has no loops} \}.\]
    The closure operation only adds points with infinite coordinates.
\end{proposition}
The condition of $M_u$ not having loops is related to the minimum being achieved twice in Equation~\eqref{eq:EqsCutOutDressian}. The set $\calL(v)^{\circ}$ has a natural polyhedral structure labelled by the initial matroids, where a cell consists of all the points in $\calL(v)$ whose associated initial matroid is constant, a given $M_u$.
By taking the closure, this polyhedral structure extends to $\calL(v)$.

\subsection{Tropicalized linear spaces}\label{subsection_trop_linear_spaces}

Let $\iota=\big[f_0:\cdots:f_n\big] \colon\P^r\hookrightarrow \P^n$ with $f_0,\dots,f_n \in (K^{r+1})^*$ be a linear embedding. 
Let $v$ be the valuated matroid of rank $r+1$ on $E = \{0, \ldots, n\}$ associated to the $f_0,\dots,f_n$,
as in Example \ref{realizableValuatedMatroids}. 
Note that another representative $f_0',\dots,f_n'$ of $ \big[f_0:\cdots:f_n\big] $ is related by multiplying with a scalar $\lambda\neq 0$. Hence its associated valuated matroid $v'$ satisfies that $v' = v + (r+1) \cdot \val \lambda$.
So both $v$ and $v'$ define the same tropical linear space $\calL(v)$ and the same underlying matroid $M$. 
Moreover, we have:

\begin{theorem}[{\cite[Proposition 4.2]{Speyer}},{\cite[Theorem B]{tropicalflagvarieties}} ]\label{thm:TropIsTrop}
The tropical linear space associated with a realizable valuated matroid coincides with the tropicalization of the corresponding linear embedding:
\begin{equation*}
 \calL(v)=\Trop(\P^r,\iota).
\end{equation*}
\end{theorem}

Without loss of generality, we may assume that $f_i \neq 0$ for all $i \in \{0,\dots,n\}$. Then $\Trop(\P^r,\iota)$ equals the closure in $\T\P^n$ of the non-compactified tropicalization
\begin{align*}
     \Trop(\P^r,\iota) \cap \R^{n+1} / \R \mathbbm{1} = \Trop\big(\iota(\P^r) \cap \G_m^n\big).
\end{align*}
Again, we can get an initial matroid $M_u$ from $u = [u_0:\cdots:u_n] \in \Trop(\iota(\P^r)\cap\G_m^n)$ by considering the matroid of rank $r+1$ on $E$ whose bases are the bases $B = \{b_0,\ldots,b_r\}$ of $M$ such that $v(B) - u_{b_0} - \ldots - u_{b_r}$ is minimal. This definition is independent of the choice of representative of $u$. 

In \cite{Rincon_localTropicalLinearSpaces} the author defines the notion of a \emph{local tropical linear space} which can be extended to our compact setting, i.e.~for tropical linear spaces in $\T\P^n$. 

\begin{definition}
\label{def:local_tropical_linear_space}
Let $B = \{b_0,\dots,b_r\}$ be a basis of $M$. The \textit{local tropical linear space}  $\Trop(\P^r,\iota)_{B} \subset \T\P^n$ is defined as the closure of the set of vectors $u \in \Trop(\iota(\P^r)\cap\G_m^n)$ such that $M_u$ contains the basis $B$.
\end{definition}

The tropical linear space $\Trop(\P^r,\iota)$ is the union of all its local tropical linear spaces $\Trop(\P^r,\iota)_{B}$.

\begin{remark} In terms of polyhedral subdivisions, the (open part) of the local tropical linear space  $\Trop(\P^r,\iota)_{B} \cap \R^{n+1} / \R \mathbbm{1}$ is a polyhedral complex dual to the faces of the regular subdivision $\mathcal{D}_v$ that contain the vertex $e_B$. For details, we refer the reader to \cite[Corollary~2.5]{Rincon_localTropicalLinearSpaces}.
\end{remark}

\subsection{The trivial valuation case}\label{section_trivvalmatroid=matroid}
Throughout this subsection, we assume that $v$ is trivially valued, i.e.~for all $A\in {E\choose r+1}$, we have that $v(A)=0$ if and only if $A$ is a basis. 
This way we may identify $v$ with its underlying matroid $M$ and the subdivision $\mathcal D_v$ of $P_M$ from Subsection~\ref{subsection_matroid_polytopes} is trivial. So the polyhedral complex of $\calL(v)$ is a fan, known as the \emph{Bergman fan} of $M$ (cf. \cite[Section 4.2]{MaclaganSturmfels} ).

The following theorem refines the polyhedral structure of the Bergman fan, by realizing its support as the order complex associated to the poset of flats of $M$.

\begin{theorem}[{\cite[Theorem 4.2.6]{MaclaganSturmfels}}]\label{theorem_trop_trivial}
    Let $M$ be a loopless matroid. The cones $\langle e_{F_1},\dots,e_{F_l} \rangle_{\R_{\geq 0}} + \R\mathbbm{1}$ in $\R^{n+1}/\mathbbm{1}$ for every chain of flats $\emptyset \subseteq F_1 \subseteq \dots \subseteq F_l$ form a fan with support $\calL(M)$. In particular, $\calL(M)\cap \R^{n+1}/\mathbbm{1} $ is homeomorphic to the cone over the order complex of the lattice of flats of $M$.
\end{theorem}

With this polyhedral structure we can describe the boundary at infinity of the cones, and see that it differs from an usual coordinate-wise compactification at infinity. 
Namely, for any cone $\sigma$ given by a chain of flats $\emptyset \subseteq F_1 \subseteq \dots \subseteq F_l$, the closure in $\T\P^n$ is given by $\langle e_{F_1},\dots,e_{F_l} \rangle_{\Rbar_{\geq 0}} + \R\mathbbm{1}$. 
In particular, if the $i$-th coordinate of a point in $\sigma$ is infinite, then we consider the minimal $j$ such that $i \in F_j$, and we have that the $i'$-th coordinate is also infinity for any $i' \in F_j$. 
Thus, there is a single maximal stratum 
in the boundary of $\calL(v)$ for each cone, which can alternatively be explained by the fact that the cone structure above triangulates the Bergman fan of $M$ (see  \cite[Section 4.2]{MaclaganSturmfels}).

\begin{example}
We consider the embedding $\P^2\xrightarrow{\id}\P^2$. Then the associated matroid has as ground set $\{e_0^{\ast},e_1^{\ast},e_2^{\ast}\}$, and every subset is independent, i.e.~we have the uniform matroid $U_{3,3}$. The Bergman fan of $U_{3,3}$ consists of a single cone $\R^2$. However, there are $6$ nontrivial flats, namely all three one-dimensional subspaces and three two-dimensional subspaces. In Figure~\ref{figure_P^2} we labelled all one-dimensional cones corresponding to non-trivial flats.

\begin{center}
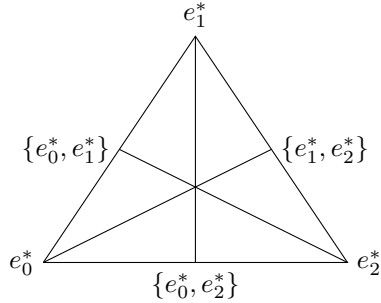
\begin{figure}[h]
\begin{tikzpicture}[scale=0.5]
\filldraw[black] (0,0) node[]{};
\draw[] (0,0) -- (4,6) -- (8,0) -- (0,0);
\draw[] (0,0) -- (6,3);
\draw[] (2,3) -- (8,0);
\draw[] (4,6) -- (4,0);

\filldraw[black] (4,6) node[anchor=south]{$e_1^{\ast}$};
\filldraw[black] (0,0) node[anchor=east]{$e_0^{\ast}$};
\filldraw[black] (8,0) node[anchor=west]{$e_2^{\ast}$};

\filldraw[black] (2,3) node[anchor=east]{$\{e_0^{\ast},e_1^{\ast}\}$};
\filldraw[black] (6,3) node[anchor=west]{$\{e_1^{\ast},e_2^{\ast}\}$};
\filldraw[black] (4,0) node[anchor=north]{$\{e_0^{\ast},e_2^{\ast}\}$};

\end{tikzpicture}
\caption{The compactified cones of $\Trop(\P^2, \id)=\T\P^2$ given by flats of the uniform matroid $U_{3,3}$. This also  represents the compactified apartment in the spherical building $\calBbar_2(K)$.}
\label{figure_P^2}
\end{figure}
\end{center}
\end{example}

\section{Faithful linear tropicalization}\label{sec:sections}

The goal of this section is to show Theorem \ref{mainthm_section} from the introduction. We recall its statement:

\begin{theorem}[Theorem \ref{mainthm_section}]
Let $\iota\colon\P^r\hookrightarrow \P^n$ be a  linear closed immersion. Then there is a natural piecewise linear embedding $J\colon \Trop(\P^r,\iota)\rightarrow \calBbar_r(K)$ that makes the following diagram commute
\begin{equation*}
\begin{tikzcd}
\calBbar_r(K)\arrow[rd,"\calBbar(\iota)"']&&\Trop(\P^r,\iota) \arrow[rd, "\subseteq"]\arrow[ll, "J"']&\\
& \calBbar_n(K)\arrow[rr,"\trop"]&& \T\P^n.
\end{tikzcd}
\end{equation*}
\end{theorem}

By a piecewise linear embedding $J$ we mean that we have a finite covering of $\Trop(\P^r,\iota)$ by subcomplexes such that the image of each subcomplex lies in an apartment and the restriction of $J$ on each subcomplex is piecewise linear. 
In particular, $\pi_{\iota}$ induces a piecewise linear homeomorphism between the union of apartments $\bigcup_{B \in \mbases M } 
 \calAbar(B)$ and the tropicalized linear subspace $\Trop(\P^r,\iota)$.

Choose $f_0,\ldots,f_n \in (K^{r+1})^\ast$ defining the embedding $\iota: \P^r \hookrightarrow \P^n$ and let $v$ be the corresponding valuated matroid of rank $r+1$ on $E = \{0,\ldots,n\}$ as in Example~\eqref{realizableValuatedMatroids} (a). 
As above, we assume that $f_i \neq 0$ for all $i \in E$. 
Let $B \in \mbases M$ be a basis of the underlying matroid $M$.
Recall that the compactified apartment $\calAbar(B)$ in $\calBbar_{r}(K)$ denotes the set of all seminorms on $(K^{r+1})^{*}$ diagonalized by~$B$. 
It is homeomorphic to~$\T\P^r$: a parametrization $\T\P^r\xrightarrow{\sim} \calAbar(B)$ is given by
\begin{equation*}
v = [v_0:\cdots:v_r] \longmapsto \seminorm{\cdot}_{B,v}. 
\end{equation*}   
This map is well defined because different projective representatives $v' = v + \lambda \allone$ give rise to homothetic seminorms  $\seminorm{\cdot}_{B,v'} =  \exp(-\lambda) ||.||_{B,v}$. 

For the proof of Theorem~\ref{mainthm_section} we need a couple of technical results first. We begin with a valuative version of Cramer's rule.

\begin{lemma}
\label{lem:Cramer}
    Let $B$ be a basis of $M$, and $k \in E \setminus B$. Write $f_k = \sum_{b \in B} \lambda_b f_b$ with $\lambda_b \in K$. For all $b \in B$ we have
    \begin{equation*}
        \val(\lambda_b) = v\big(k \cup B \setminus b\big) - v(B).
    \end{equation*}
\end{lemma}

\begin{proof} 
Label the elements of $B$ as $b_0, \dots, b_r$, with $b_0$ equal to our chosen $b$.
Using multilinearity of the determinant and properties of the valuation we find:
\begin{align*}
   v(k \cup B \setminus b) &= \val\Big(\det\begin{bmatrix} f_k & f_{b_1} &\cdots &f_{b_r}\end{bmatrix}\Big)\\  
   &= \val\Big(\sum_{i \in B} \lambda_i \det\begin{bmatrix}f_i & f_{b_1} &\cdots &f_{b_r}\end{bmatrix}\Big)\\ 
   &= \val\Big(\lambda_{b_0} \cdot \det\begin{bmatrix}f_{b_0}& \cdots & f_{b_r}\end{bmatrix}\Big)\\ 
   &= \val(\lambda_b) + v(B).  \qedhere
\end{align*}
\end{proof}

\begin{lemma}
\label{le:BasisHasOneFiniteCoord}
    Let $B$ be a basis of $(E, v)$. 
    If $u$ is in $\calL(v)$,
    then there is $k \in B$ such that $u_k \ne \infty$.
\end{lemma}

\begin{proof}
    If $B=E$, the statement follows from the definition of tropical projective space.
    Otherwise, assume there is $u \in \calL(v)$ with $u_i=\infty$ for all $i \in B$. 
    Choose $k \in E\setminus B$ such that $u_k<\infty$ and set $\tau= k \cup B$. 
    Since $B$ is a basis, we have that $v(B)+u_b<\infty$. 
    But this is the only finite term in the minimum for Definition~\ref{definition_tropicallinearspace} hence it is not attained twice for $\tau$, which is a contradiction.
\end{proof}

For the following statement we interpret \cite[Theorem 2.6]{Rincon_localTropicalLinearSpaces} in our setting and extend to the compactifications. 
It gives a piecewise linear homeomorphism between the local tropical linear space $\Trop(\P^r,\iota)_{B}$ (as in Definition \ref{def:local_tropical_linear_space}) and the compactified apartment $\calAbar(B)$.

\begin{proposition}\label{prop_section_local}
The map $J_{B}$ sending $u \in \Trop(\P^r,\iota)_{B}$ to the seminorm $\seminorm \cdot_{B,u_B}$, where $u_{B} = u \cdot e_B$, is a piecewise linear homeomorphism between $\Trop(\P^r,\iota)_{B}$ and~$\calAbar(B)$.
Its inverse is the restriction of $\pi_{\iota}$. Explicitly, the seminorm $x = \vert \vert.\vert\vert_{B,v} \in \calAbar(B)$ is mapped to $\pi_{\iota}(x) \in \Trop(\P^r,\iota)_{B}$ with 
\begin{align*}
    \pi_{\iota}(x)_k = \left\{
\begin{array}{ll}
v_k & \textrm{ if } k \in B, \\
\min_{i \in B} v(k \cup B \setminus i) - v(B) + v_k & \textrm{ otherwise. }\\
\end{array}
\right.
\end{align*}
\end{proposition}

\begin{proof}
First, we define $J_{B}$ on the open dense subset $\Trop(\P^r,\iota)_{B} \cap \R^{n+1}/\R\mathbbm{1}$ and show that this gives a piecewise linear homeomorphism to $\mathcal{A}(B)$.
Clearly, the map $\pi_{\iota}$ is injective on $\calA(B)$. 
So we show that $\pi_{\iota} \circ J_{B}$ is the identity on $\Trop(\P^r,\iota)_{B} \cap \R^{n+1}/\R\mathbbm{1}$,
that is for $u$ in the latter set we show

\[
\pi_\iota\big(J_B(u)\big) = \pi_{\iota}\big(\vert\vert.\vert\vert_{B,u_{B}}\big) = \big[-\log \vert\vert f_{0} \vert\vert_{B,u_{B}}:\cdots:-\log \vert\vert f_{n} \vert\vert_{B,u_{B}}\big] = [u_0:\cdots:u_n].
\]
Since $-\log \seminorm{ f_b}_{B,u_{B}} = u_b$ for $b \in B$, we only need to check the equality for $k$ in $E \setminus B$. 
As $u \in \Trop(\P^r,\iota)$, the condition in Definition \ref{definition_tropicallinearspace} for $B \cup \aset k$ says that the minimum is attained at least twice in 
  \begin{equation*}
      \min_{i \in k \cup B} \big\{v(k \cup B \setminus i) + u_{i} \big\}.
  \end{equation*}
  By subtracting $u \cdot e_{k \cup B}$ we see it is equivalent to the minimum being attainted at least twice in
    \begin{align*}
      \min_{i \in k \cup B} \big\{v(k \cup B \setminus i ) - u \cdot e_{k \cup B \setminus i} \big\}.
  \end{align*}
  Since $B$ is a basis of the initial matroid $M_u$, it minimizes the expression $v(\sigma) - u \cdot e_\sigma$ over all $\sigma \in {E \choose {r+1}}$. 
  So the minimum is achieved at $i = k$ and some other $i = l$. 
  That is, for all $i \in B$ we have
  \[
      v(B) + u_k = v\big(k \cup B \setminus l\big) + u_{l} \leq v\big(k \cup B \setminus i\big) + u_i.
    \]
Writing $f_k = \sum_{i \in B} \lambda_i f_i$,
by Lemma \ref{lem:Cramer} this is equivalent to
\[
    u_k = \val(\lambda_l) + u_{l} \leq \val(\lambda_i) + u_i
  \]
  for all $i \in B$. 
  Hence, $-\log \seminorm{f_k}_{B,u_{B}} = \min_{i \in B} \big(\val(\lambda_i) + u_{i}\big) = u_{k}$.

This also gives the alternative description of $\pi_{\iota}$. We are left to show that ${\pi_{\iota}}\big(\mathcal{A}(B)\big)$ lies in $\Trop(\P^r,\iota)_{B} \cap \R^{n+1} / \R \mathbbm{1}$ for which we refer to \cite[Theorem 2.6]{Rincon_localTropicalLinearSpaces}.
Hence, $J_{B}$ and the restriction of $\pi_{\iota}$ are piecewise linear inverse homeomorphisms between $\Trop(\P^r,\iota)_{B} \cap \R^{n+1}/\R\mathbbm{1}$ and $\calA(B)$. 
We now extend $J_{B}$ naturally to a piecewise linear map $\Trop(\P^r,\iota)_{B} \longrightarrow \calAbar(B)$ by sending $u = [u_0:\ldots:u_n]$ to $\vert \vert.\vert\vert_{B,u_B}$. 
This is well-defined if $u \cdot e_B$ lies in $\T\P^r$, namely if there is at least one finite coordinate in $u \cdot e_B$; this is proven below in Lemma~\ref{le:BasisHasOneFiniteCoord} and using the fact that  $\Trop(\PP^r, \iota) = \calL(v)$.
Then, $\pi_{\iota} \circ J_{B}$ is the identity on $\Trop(\P^r,\iota)_{B}$ since we have seen it is the identity for a dense subset and $\Trop(\P^r,\iota)_{B}$ is Hausdorff. 
Since $\calA(B)$ is dense in $\calAbar(B)$, 
the image $\pi_{\iota}\big(\calAbar(B)\big)$ is exactly the closure of $\Trop(\P^r,\iota)_{B} \cap \R^{n+1}/\R\mathbbm{1}$ which is $\Trop(\P^r,\iota)_{B}$. As before, $J_B \circ \pi_{\iota}|_{\calAbar(B)} = \id$ on $\calAbar(B)$ which concludes the proof.
\end{proof}

\begin{proof}[Proof of Theorem \ref{mainthm_section}]
Recall that  $\Trop(\P^r,\iota)$ is the union of all its local tropical linear spaces $\Trop(\P^r,\iota)_{B}$ where $B$ runs over the bases of the matroid $M$ associated to~$\Trop(\P^r, \iota)$. 
Thus, we define $J$ locally as $J_B$ and show that the maps $J_{B}$ glue. 
For $A, B \in \calB(M)$, we show that $J_{A}$ and $J_{B}$ glue on the open part $\Trop(\P^r,\iota)_{A} \cap \Trop(\P^r,\iota)_{B} \cap \R^{n+1} / \R \mathbbm{1}$. 
Choose $u=[u_0:\dots:u_n]$ in the latter set and, as in Proposition~\ref{prop_section_local}, write $u_A, u_B$ for the vectors in $\R^{r+1} / \R \mathbbm{1}$ with the coordinates of $u$ indexed by $A$ and~$B$, respectively. 
If $\seminorm{\cdot}_{A,u_{A}}$ is homothetic to
$ \seminorm{\cdot}_{B, u_B}$, we get a piecewise linear homeomorphism between $\Trop(\P^r,\iota)_{A} \cap \Trop(\P^r,\iota)_{B} \cap \R^{n+1} / \R \mathbbm{1}$ and $\calA(A) \cap \calA(B)$.

\textbf{Claim: } The seminorm $\seminorm{\cdot}_{A,u_A}$ is homothetic to $\seminorm{\cdot}_{B, u_B}$.

First assume that $A$ and $B$ differ by two elements, i.e.~there are $a \in A$ and $b \in B$ such that $A \setminus a = B \setminus b$.
Choose $v \in (K^{r+1})^{\ast}$ and write it in terms of the bases $A$ and $B$:
let $v = \sum_{i \in A} \alpha_i f_i = \alpha_a f_a + \sum_{i \in A \setminus a} \alpha_i f_i$ 
and $f_a = \beta_b f_b + \sum_{j \in B \setminus b} \beta_j f_j$.
Replacing, we get $v = \alpha_a \beta_b f_b 
+ \alpha_a \sum_{j \in B \setminus b} \beta_j f_j  
+ \sum_{i \in A \setminus a} \alpha_i f_i
= \alpha_a \beta_b f_b +  \sum_{j \in B \setminus b} (\alpha_a  \beta_j   + \alpha_j) f_j$.
We want to show equality of the expressions
  \begin{align}
  \label{eq:log_AuA_BuB}
  -\log \seminorm{v}_{A,u_{A}} &=  \min \left \{ \val(\alpha_a) + u_{a},  \min_{i \in A \setminus a}{  \val(\alpha_i) + u_{i}} \right \}, \\
\nonumber     -\log \vert\vert v \vert\vert_{B,u_{B}} 
     &= \min \left\{ \val(\alpha_a) + \val(\beta_b) + u_b, \min_{j \in B \setminus b} \{\val(\alpha_a \beta_j + \alpha_j) + u_{j} \} \right\}.
  \end{align}
By Lemma~\ref{lem:Cramer} we have for every $j \in B$ that  
\begin{align}
\label{eq:beta_i}
      \val(\beta_j) = v(a \cup B \setminus j) - v(B) = v(A \cup B \setminus j) - v(B).
\end{align}

As $u \in \Trop(\P^r,\iota)_B$, 
the basis $B$ is in the initial matroid $M_u$, 
so $B$ minimizes the expression $v(\sigma) - u \cdot e_\sigma$ over all $\sigma \in {E \choose r+1}$, 
so $v(B) - u \cdot e_B \le v(A \cup B \setminus j) - u \cdot e_{A \cup B \setminus j}$ for all $j \in B$. 
By Equation~\eqref{eq:beta_i} we get
\begin{align} \label{eq:beta_u} u_a - u_j = u \cdot (e_{A \cup B \setminus j} - e_B) \le v(A \cup B \setminus j) - v(B) = \val(\beta_j). \end{align}
Also $u$ is in $\Trop(\P^r,\iota)_A$, 
so $A$ is in the initial matroid $M_u$, 
thus $ v(A) - u \cdot e_A \le v(B) - u \cdot e_B$ and so
\begin{align} \label{eq:beta_b} u_a - u_b \ge v(A) - v(B) = \val(\beta_a). \end{align}
If we set $j$ equals $b$ in Equation~\eqref{eq:beta_u} and combine it with Equation~\eqref{eq:beta_b} we get $\val(\beta_b) + u_b = \val(\beta_a) + u_a$, and furthermore $\val(\alpha_a) + \val(\beta_b) + u_b = \val(\alpha_a) +  u_a$. So equality in the first terms of Equation~\eqref{eq:log_AuA_BuB} happens.
  
It remains to show for $i \in A \setminus a$ that either $\val(\alpha_a \beta_i + \alpha_i) = \val(\alpha_i)$ or $ \val(\alpha_i) + u_i \ge  \val(\alpha_a \beta_i + \alpha_i) + u_i\ge \val(\alpha_a) + \val(\beta_b) + u_b$. 
By properties of valuations we have
\[\val(\alpha_a \beta_i + \alpha_i) \ge \min\big(\val(\alpha_a) + \val(\beta_i), \val(\alpha_i)\big), \]
and moreover if $ \val(\alpha_a) + \val(\beta_i) > \val(\alpha_i)$, then $\val(\alpha_a \beta_i + \alpha_i) = \val(\alpha_i)$  and we are done. 
Thus, assume that $ \val(\alpha_a) + \val(\beta_i) \le \val(\alpha_i)$. 
In that case, we calculate using Equations~\eqref{eq:beta_i} and \eqref{eq:beta_b}, the following
\begin{align*}
    \big(\val(\alpha_a \beta_i + \alpha_i) + u_i\big) -\big(\val(\alpha_a) + \val(\beta_b) + u_b\big)
    &\ge \val(\beta_i) - \val(\beta_b)  + u_i - u_b  \\
    &\ge u_a - u_b - \val(\beta_b) \ge 0.
\end{align*}
Thus, both minima in Equation~\eqref{eq:log_AuA_BuB} coincide. 
For the general case where $A \Delta B$ has $2m$ elements, since both $A$ and $B$ are bases in the initial matroid $M_u$, by the basis exchange axiom there is a sequence of bases $B_0, B_1, \dots, B_m$ such that $B_0 = A$ and $B_m = B$ and every pair $B_q$, $B_{q+1}$ differs by two elements. 
Thus, we may apply our previous argument to every pair $B_q, B_{q+1}$ to conclude the claim for inner points of the local maps. 
Again, extending to the compactifications concludes the proof.
\end{proof}


\section{Examples: The trivially and the discretely valued case}\label{section_examples}
\subsection{The trivial valuation case}

We can make Theorems \ref{mainthm_limits} and \ref{mainthm_section} explicit when the valuation is trivial. 
Recall from Example \ref{ex_trivialvaluationdescription}, that there is a bijection
$$\calBbar(V)\stackrel{1:1}{\longleftrightarrow} \big\{\big(0=V_0 \subsetneq  V_1 \subsetneq  \dots \subsetneq  V_l=V^{\ast},c_1>\dots>c_{l-1}\big)\ \big\vert \   c_1,\dots,c_{l-1} \in \Rbar_{> 0} \big\}_{l=1,\ldots,r+1}.$$
In other words, a class of a seminorm is given by a flag of subspaces together with decreasing coordinates corresponding to logarithms of the (constant) values of the representative with generic value $1$. We fix an embedding $\iota=[f_0:\ldots:f_n]: \P^r\to \P^n$ and obtain a realizable matroid $M$ on $[n]$. Recall that we can compute $\Trop(\P^r,\iota)$ as a compactification of the cone complex over the order complex of flats of $M$, cf.\ Theorem \ref{thm:TropIsTrop} and Theorem \ref{theorem_trop_trivial}.

\subsubsection*{Theorem \ref{mainthm_limits}}
We can explicitly compute the maps $\pi_{\iota}$ in terms of both the coordinates of the building above and the description and coefficients for flats of $M$. \\
Let $ x\in \calBbar_r(K)$ be given by a flag $0=V_0 \subsetneq  V_1 \subsetneq  \dots \subsetneq  V_l=(K^{r+1})^{\ast}$ and coordinates $c_1>\dots>c_{l-1}$ as in Example \ref{ex_trivialvaluationdescription}. We formally set $c_0:=\infty,\ c_l:=0$. Let $||.|| \in x$ be a representative with generic value $1$ and $d_j:=\exp(-c_j)$. Recall that then $d_j$ is the constant value of $||.||$ on $V_j$ for $j=0,\dots,l$. Fix a given coordinate $i \in [n]$ and let $j$ be minimal such that $f_i\in V_j$. Then $-\log\big(||f_i||\big)=-\log(d_j)=c_j$. We consider the matroid on $[n]$ induced by $\iota$ and for $j=0,\dots,l$ we define a flat $F_j=\{i \in [n] \mid f_i \in V_j\}$. Then the above computation shows
\begin{align*}
\pi_{\iota}\big(||.||\big)  &= \big[-\log(||f_0||):\cdots : -\log(||f_n||)\big] \\
&=  \sum_{j=0}^{l-1} (c_j-c_{j+1}) e_{F_j}\in \Trop(\P^r,\iota).
\end{align*}

\subsubsection*{Theorem \ref{mainthm_section}}

A point $u=[u_0:\dots:u_n]\in \Trop(\P^r,\iota)$ can be written as 
$$u=\sum_{j=0}^l a_j e_{F_j}$$
for $a_j \in \R_{\geq 0}$ and a chain of flats $F_0\subseteq\dots \subseteq F_l = [n]$. Note that each flat $F_j$ of $M$ yields a subspace $V_j$ of $(K^{r+1})^{\ast}$ and that proper inclusions of flats yield proper inclusions of their corresponding subspaces.\\
Let $B \in \scrB(M)$ be any basis such that for all $j=0,\dots,l$ the rank of $F_j\cap B$ equals the rank of $F_j$. Such a basis can be obtained by successively extending bases of the flats. Then $u$ lies in the local tropical linear space $\Trop(\P^r,\iota)_{B}$. We want to compute the map $J_{B}: \Trop(\P^r,\iota)_{B}\to \calBbar_r(K)$ from Proposition \ref{prop_section_local}.\\
Fix $i \in B$ and let $j$ be minimal such that $i\in F_j$. Then we have
\begin{equation*} 
u_i=\sum_{k: i\in F_k}a_k 
=\sum_{k=j}^l a_k
\end{equation*}
Then $J_{B}$ sends $u$ to the homothety class of the seminorm having generic constant value $\exp(-u_i)=\exp(-\sum_{k=j}^l a_k)$ on $V_j$. In the coordinates from Example \ref{ex_trivialvaluationdescription}, if we set 
$c_j := \sum_{k=j}^{l-1} a_k$, we have
$$J(u)=J_{B}(u)=\big(0\subsetneq V_1\subsetneq \dots\subsetneq V_l=(K^{r+1})^{\ast}, c_1>\dots>c_{l-1}\big).$$

\subsection{Lattices and the discrete valuation case}\label{sec:lattices}
Let $V$ be a vector space of dimension $r+1$ over $K$.
\begin{definition}
    A \emph{lattice} in $V^*$ is an $\OO_K$-submodule $\Lambda$ such that
    \[ \Lambda\otimes_{\OO_K}K\simeq V^*.\]
   To a linear map $f\colon V\to W$ and a lattice $\Lambda\subseteq V^*$ we can associate the lattice $(f^t)^{-1}(\Lambda)$ in $W^*$.
\end{definition}

\begin{remark}
    That $\Lambda\otimes_{\OO_K}K\simeq V^*$ is equivalent to the following: for every $f\in V^*$ there exists a $c$ in $K$ such that $c^{-1}f$ is in $\Lambda$.
\end{remark}
Note that, despite $K$ being generated over $\OO_K$ by infinitely many elements $e_{-\gamma},\gamma\in\Gamma_{\geq 0}$ (where $\val(e_{-\gamma}) = -\gamma$), these generators satisfy obvious relations $e_{-\gamma_1}=r_{12}e_{-\gamma_2}$ with $r_{12}\in\OO_K$ for $\gamma_1<\gamma_2$. When inverting all the $r_{ij}$, the relations show that the resulting module is freely generated by one element $e_0$, which can be taken to be $1\in K$, showing that $K\otimes_{\OO_K}K\simeq K$. So a lattice could contain a vector subspace of $V^*$, in which case it would not be finitely generated as an $\OO_K$-module.
  \begin{lemma}  If a lattice is finitely generated, then it is free on $r+1$ generators. \end{lemma}
  \begin{proof} Since the lattice spans a vector space of dimension $r+1$ over $K$, $r+1$ is clearly the minimum number of generators. Suppose that there were $r+2$. Then we would find a non-trivial $K$-linear dependence relation. Compare the denominators and multiply by the one (say it is the $0$th) achieving the highest absolute value. We thus obtain an $\OO_K$-linear dependence relation with invertible $0$th coefficient, showing that the $0$th generator is redundant.
  \end{proof}
In the following we present ways to go between seminorms and lattices. In general, they are not inverse to one another (cf. \cite[Lemma I.2.2]{Schneider}).
\begin{definition}
    Let $\Lambda\subseteq V^*$ be a lattice. The seminorm associated to $\Lambda$ is (called its \emph{gauge}):
    \[q_\Lambda(f)=\inf_{f\in c\Lambda}\lvert c\rvert\in\Rbar.\]
\end{definition}
\begin{definition}
    The lattice associated to a seminorm $q$ is the closed unit ball $\Lambda_q=q^{-1}([0,1])$.
\end{definition}
\begin{remark}
     If $\Lambda$ contains a vector subspace of $V^*$, the associated seminorm is not a norm, and vice versa.
\end{remark}
Note that all norms on a finite-dimensional vector space are equivalent. In particular, the space is complete with respect to any norm. Once a basis $V\simeq K^{r+1}$ is chosen, one such norm is given by $\lvert\lvert(\lambda_0,\ldots,\lambda_r)\rvert\rvert_\infty=\max_i\lvert \lambda_i\rvert$, which is clearly diagonalizable. The \emph{closed unit polydisc} is the lattice $\{v\in V\colon\lvert\lvert v\rvert\rvert_\infty\leq1\}$. It is finitely generated by the elements of the chosen basis.
\begin{proposition}         
        If $K$ is spherically complete, the correspondences above induce an equivalence between closed unit polydiscs (with respect to some basis and $\lvert\lvert\cdot\rvert\rvert_\infty$) and $\Gamma$-valued norms.
\end{proposition}
\begin{proof}
    
Let $q$ be a $\Gamma$-valued norm. Since $K$ is spherically complete, there exists a basis $\{v_0,\ldots,v_r\}$ of $V^*$ such that $q(\lambda_0v_0+\ldots+\lambda_rv_r)=\max_i\big(\lvert\lambda_i\rvert\alpha_i\big)$. Note that $\alpha_i=q(e_i)$ is an element of the value group by assumption. Let $c_i\in K$ be any element such that $\lvert c_i\rvert=\alpha_i$. Then the associated lattice $\Lambda_q$ is the closed unit polydisc with respect to the basis $\{c_0^{-1}v_0,\ldots,c_r^{-1}v_r\}$.

Vice versa, if $\Lambda$ is the closed unit polydisc with respect to a basis $\{v_0,\ldots,v_r\}$, then $q_{\Lambda}=\lvert\lvert\cdot\rvert\rvert_\infty$ with respect to the same basis. Note that $v\in c\Lambda_q$ if and only if $q(v)\leq\lvert c\rvert$, therefore
\[q_{\Lambda_q}(v)=\inf\big\{\lvert c\rvert \ \big\vert\  v\in c\Lambda_q\big\}=q(v),\]
since we have assumed that $q$ takes values in $\Gamma$.
The inclusion $\Lambda\subseteq\Lambda_{q_\Lambda}$ is always true. On the other hand, $v\in\Lambda_{q_\Lambda}$ if and only if $q_{\Lambda}(v)\leq 1$; but since $q_{\Lambda}=\lvert\lvert\cdot\rvert\rvert_\infty$ with respect to some basis for which $\Lambda$ is the closed unit polydisc, it is clear that $q_{\Lambda}(v)\leq 1$ if and only if $v\in\Lambda$.
\end{proof}

\begin{remark}
    In the spherically complete case, a description of all lattices can be found in \cite[Theorem 3.6]{ChernikovMennen}.
\end{remark}

\begin{corollary}
    If $K$ is a complete discretely valued field, the correspondences above induce an equivalence between finitely generated lattices and integer-valued norms.
\end{corollary}

Since we are interested in the building of $PGL$, we consider (semi)norms up to homothety ($q\sim\gamma q$ for any $\gamma\in\exp(\Gamma)$), which correspond to lattices up to homothety ($\Lambda\sim c\Lambda$ for any $c\in K^\ast$).

\smallskip

\paragraph{\emph{Simplices}}
In the discretely valued case, there is a way of reconstructing the simplicial structure of the building in terms of nested sequences of finitely generated lattices and collections of real numbers.

Indeed, to any norm $q$ we can associate a nested sequence of lattices        
\[\big\{\Lambda(c)=q^{-1}([0,c])\ \big|\ c\in [1,e]\big\}.\]
Up to homothety, this list consists of $0<k+1\leq r+1$ lattices $\Lambda_0\subseteq\ldots\subseteq\Lambda_k\subseteq\pi^{-1}\Lambda_0$, where $\pi$ is the uniformiser of $\OO_K$. Moreover, we can associate to $q$ the list of jumps:
\[\big\{c_i=\inf\{c\in(1,e]\ \mid\ \Lambda(c)=\Lambda_i\}\big\}_{i=1,\ldots,k}.\]

Vice versa, given a nested sequence of lattices $\Lambda_0\subseteq\ldots\subseteq\Lambda_k\subseteq\pi^{-1}\Lambda_0$, we can find a basis $\{e_0,\ldots,e_r\}$ of $\Lambda_k$ such that 
\[\Lambda_{h-1}=\langle\pi e_0,\ldots,\pi e_{i_h},e_{i_h+1},\ldots,e_r\rangle\]
for some $0\leq i_k<\ldots<i_1<r=:i_0$. Given $c_1,\ldots,c_k\in(1,e)$, we set $c_0=1$ and $\alpha_i=c_j$ if $i_{j}<i\leq i_{j-1}$. We then define the associated norm 
\[q(\lambda_0e_0+\ldots+\lambda_re_r)=\max_i\big\{\lvert\lambda_i\rvert\alpha_i\big\}.\]

\begin{remark}
    Possibly infinitely-generated lattices in $V^*$ are dual to finitely generated submodules of possibly non-maximal rank in $V$, which provides another description of the compactification of the building, as explained in \cite[\S\S 3-4]{Werner_seminorms} in the case of a local field.
\end{remark}
 
 \paragraph{\emph{Convexity}}
    In the discretely-valued case, there are various notions of convexity in the building: Weyl convexity (see for instance \cite[\S 4.11]{AbramenkoBrown}) is the one that bears the closest resemblance to the metric approach to buildings. There are weaker notions of convexity that have been studied in \cite{JoswigSturmfelsYu} from a tropical perspective: a set of lattices up to homothety is $+$-convex (resp. $\cap$-convex) if it is closed under rescaling and taking sums (resp. intersections) as submodules of $K^{r+1}$. In terms of norms, these operations correspond to taking pointwise maximum (resp. the largest norm that is bounded above by pointwise minimum).
    
    \emph{Membranes}, introduced in \cite{KeelTevelev}, are $+$-convex unions of apartments, consisting of all lattices admitting a basis of the form $\{\pi^{a_0}f_{i_0},\ldots,\pi^{a_r}f_{i_r}\}$, where the $a_i$'s are integers and the $f_i$'s are chosen from a fixed set of $n+1$ vectors $\{f_0,\ldots,f_n\}$ in $K^{r+1}$. In \cite[Theorem 4.11]{KeelTevelev} and \cite[Theorem 18]{JoswigSturmfelsYu}, the authors show that the lattice points in the membrane $[M]$ correspond bijectively to integer points in $\Trop(\P^r,\iota)$, where $\iota$ denotes the embedding  $[f_0:\ldots:f_n]: \P^r \xrightarrow{} \P^n$. Both the membrane and the tropical linear space are indeed the tropical convex hull of finitely many points (at infinity). The correspondence is based upon the \emph{nearest point map} onto the tropical lattice polytope $\Trop(\P^r,\iota)$ described in \cite[Lemma 21]{JoswigSturmfelsYu}, which can be interpreted as the tropicalization map $\pi_{\iota}$ and the section $J$ (see Section \ref{sec:sections} for the definitions) restricted to the affine part $\Trop(\P^r,\iota) \cap \R^n / \R \mathbbm{1}$.


\section{The universal realizable valuated matroid}

\subsection{Infinite tropicalization}\label{section_infiniteGrassmannian}

Let $v$ be a valuated matroid on a, possibly infinite, ground set $E$, and let $E'\subset E$ be any subset containing a basis. Then the restriction of $v$ to ${E' \choose r+1}$ yields again a valuated matroid $v'$. 
We associate a tropical linear space to $v$ in general, by gluing together the usual construction for finite valuated matroids.
We define the sets
\begin{align*}
    \T\P^E&:=\left(\{(u_e)_{e\in E} \mid u_e \in \Rbar \}\setminus \{(\infty)_{e \in E}\}\right)/ \R\mathbbm{1} \\
    U_{v}&:=\big\{(u_e)_{e\in E}  \in \T\P^E \ \big\vert\ \textrm{ for all bases } A\subset E \textrm{ there is } a \in A \textrm{ with } u_a \neq \infty \big\} \subset \T\P^E.
\end{align*}

\begin{definition}\label{def:trop_v}
   The \emph{tropical linear space} $\calL(v)\subset\TT\PP^E$ associated to $v$ is the set of $(u_e)_{e \in E} \in \T\P^E$ such that for any $\tau\in{E \choose r+2}$ the minimum in $v\big(\tau\setminus\{e\}\big) + u_e$ is attained at least twice.
\end{definition}

Note that the proof of Lemma~\ref{le:BasisHasOneFiniteCoord}, does not use the finiteness condition of $E$, hence we have for a valuated matroid $v$ and its associated tropical linear space that
    \[\calL(v) \subseteq U_v\subseteq\T\P^E.\]
Passing to the smaller set $U_v$ allows us to endow it with a limit topology as follows. Let $I$ be the category of finite subsets $E'$ of $E$ containing a basis, with inclusions as morphisms. Then we have a functor from $I^{\textrm{op}}$ into the category of topological spaces, assigning to each $E'$ the space $U_{v'}\subseteq \T\P^{E'}$, where $v'$ is the restriction of $v$ to $E'$,
and to every inclusion the corresponding coordinate projection. Note that these coordinate projections are well-defined by the construction of the $U_{v'}$. 
We see that

$$U_v= \varprojlim_{E' \in I} U_{v'},$$
and we can endow it with the limit topology.
Since we also have an identification
$$\calL(v)\xlongrightarrow{\sim}\varprojlim_{E' \in I} \calL(v'),$$
and in particular we can endow $\calL(v)$ with the limit topology.

We now restrict our attention to realizable valuated matroids. Recall that the \emph{universal realizable matroid} $w_{\rm univ}$ is given by $w_{\rm univ}(A)=\val\big(\det(A)\big)$ for $A\in{E\choose{r+1}}$. 
Theorem \ref{mainthm_limits} and Speyer's result on the tropicalization of finite linear subspaces (Theorem \ref{thm:TropIsTrop}) allow us to identify the space of seminorms on $(K^{r+1})^{\ast}$ up to homothety with the tropical linear space associated to the universal realizable matroid. \begin{theorem}[Theorem \ref{maintheorem_infinite_tropicalization}] 

The Goldman-Iwahori space is the tropical linear space associated to the universal realizable matroid $w_{\univ}$, i.e.
$$\calXbar_r(K)=\calL(w_{\rm univ}).$$
\end{theorem}
\begin{proof}
By Theorem \ref{mainthm_limits} and Remark \ref{rem_different_categories} (c), we can write $\calXbar_r(K)$ as the limit of all linear tropicalizations with respect to the category of coordinate projections. The tropicalization functor from this category is naturally equivalent to the functor which associates to an embedding $\iota=[f_0:\ldots:f_n]$ the tropical linear space associated to the valuated matroid given by $\{f_0,\dots,f_n\}$, as repeating entries and permuting coordinates yields homeomorphic tropicalizations.
\end{proof}

Let $E$ denote the set of (non-zero) vectors in $K^{r+1}$. We obtain linear maps
\begin{equation}\label{universal_projection}
\bigoplus_E K\twoheadlongrightarrow K^{r+1},
\end{equation}
and dually
\[\iota_{\rm univ}\colon K^{r+1}\hooklongrightarrow K^E.\]
As in Example \ref{realizableValuatedMatroids} and Section \ref{subsection_trop_linear_spaces}, we may associate to $\iota_{\rm univ}$ the realizable valuated matroid $w_{\rm univ}$. Hence, we can interpret $\calXbar_r(K)$ as the tropicalization of the universal projective linear subspace of rank $r$.

In the following we will show that $\calXbar_r(K)$ is cut out by much simpler equations than the ones coming from the universal realizable valuated matroid.
We can think of (\ref{universal_projection}) as the $(r+1)\times E$ matrix whose $e$-th column vector represents $e$ in the standard basis of $K^{r+1}$. It follows that the $i$-th row corresponds to the $i$-th coordinate projection as a function on $E$, that is:
\begin{proposition} \label{lem:small_circuits}
 The image of $\iota_{\rm univ}\colon K^{r+1}\hookrightarrow K^E$ consists of (the restrictions of) all the linear maps from $K^{r+1}$ (resp. $E$) to $K$. In particular, the equations of $\iota_{\rm univ}$ involve only finitely many variables.
\end{proposition}
\begin{proof}
 Let $x_e$ denote the coordinate on $K^E$ such that $x_e(f)=f(e)$. The equations of $\iota_{\rm univ}$ are
 \begin{equation*}\begin{split}
  x_{\lambda e}=\lambda x_e,\quad \quad &\textrm{ for } \lambda\in K^\times,e,\lambda e\in E;\\ x_{e_1+e_2}=x_{e_1}+x_{e_2},\quad \quad  &\textrm{ for }e_1,e_2,e_1+e_2\in E. \qedhere
  \end{split}\end{equation*}
\end{proof}
\begin{remark}
 These are the equations of $K^{r+1}\hookrightarrow K^{E'}$ for any subset $E'\subseteq E$ and the corresponding projection $K^E\to K^{E'}$ (restriction of functions).
\end{remark}

Given a basis $(e_1,\ldots,e_{r+1})$ of $K^{r+1}$ (e.g.~the standard one), these equations are equivalent to
\[x_e=\sum_{i=1}^{r+1}\on[e]_ix_{e_i}.\]

As a curiosity, we note that the \emph{large} circuits of 
Definition \ref{definition_tropicallinearspace} are equivalent to the tropicalization of the \emph{small} circuits from Proposition \ref{lem:small_circuits}. 
\begin{proposition}
Write $w$ for $w_{\univ}$. The minimum is attained at least twice in all
 \begin{align}
  \min\big(u_{\lambda e},\, u_e + \val(\lambda)\big)\qquad &\textrm{ for }\lambda\in K^\times \textrm{ and } e,\lambda e\in E;\label{eqn:small_circuits_i}\\
  \min\big(u_{e_1+e_2},\, u_{e_1},\, u_{e_2}\big) \qquad &\textrm{ for } e_1,e_2,e_1+e_2\in E.     \label{eqn:small_circuits_ii}
  \end{align}
  if and only if the minimum is attained at least twice in all
  \begin{equation}\label{eqn:all_circuits}
   \min_{i \in \tau}\big(u_{i} + w(\tau \setminus i)\big)\qquad \textrm{ for }\tau\in{E\choose{r+2}}.
  \end{equation}
\end{proposition}
\begin{proof}
 \eqref{eqn:all_circuits}$\Rightarrow$\eqref{eqn:small_circuits_i}: 
 If $e$ is not $0$, consider a basis $B=\{e,e_1,\ldots,e_r\}$ and apply \eqref{eqn:all_circuits} to $\{\lambda e\}\cup B$. 
 Note that $w(\tau \setminus i  ) = \infty$ unless $i=e,\lambda e$. Equation \eqref{eqn:small_circuits_i} follows. Equation \eqref{eqn:small_circuits_ii} follows similarly.
 
\eqref{eqn:small_circuits_i} and \eqref{eqn:small_circuits_ii}$\Rightarrow$\eqref{eqn:all_circuits}: 
From Equation~\eqref{eqn:small_circuits_i} we get $u_{\lambda e}=u_e+\val(\lambda)$, and by induction from Equation~\eqref{eqn:small_circuits_ii} we get that for $k\ge 2$ the minimum is attained at least twice in 
\begin{align}\label{eqn:sums}
\min( u_{e_1+\ldots+e_k}, u_{e_1}, \dots, u_{e_k}).
\end{align}
Let $\tau\in{E\choose{r+2}}$. 
If $\tau$ contains no basis, then Equation~\eqref{eqn:all_circuits} is trivially true. Thus, suppose $\tau = f \cup B$, with $B$ a basis, and write $f =\sum_{e \in B}\lambda_e e$.
By Lemma~\ref{lem:Cramer}, we have that $\val(\lambda_i) = w(\tau \setminus i ) - w(B)$ for all $i \in B$. If we set $\lambda_f = 1$, so  $\val(\lambda_f) = 0$ and  $\val(\lambda_i) = v(\tau \setminus i ) - w(B)$ also for $i = f$, we get
\begin{align*} 
  \min_{i \in \tau}\big(u_{i} + w(\tau \setminus i)\big) - w(B) =& \min_{i \in \tau}\big(u_{i} + w(\tau \setminus i) - w(B)\big)\\
  =& \min_{i \in \tau}\big(u_{i} + \val(\lambda_i)\big) = \min_{i \in \tau}(u_{\lambda_i i}),
\end{align*}
and the conclusion follows from Equation~\eqref{eqn:sums}.
\end{proof}

\subsection{Tight spans} \label{section_tightspan} 
Many of our results extend results by Dress and collaborators in T-theory~\cite{dress1996t}. 
A tight span is an isometric embedding of a metric space $E$ into a hyperconvex metric space~$T_E$. 
The motivation for these spaces is fitting phylogenetic data; see the cited work for a discussion. 
There are also applications to extending valuations in $p$-adic geometry \cite{DressTerhalle_padic}.
A so-called four-point condition \cite[Section 4.6]{dress1996t}  is necessary and sufficient for a tight span to be an $\RR$-tree; see Figure~\ref{fig:example_trivial_valuation} for an example of an $\RR$-tree.
This condition is essentially the basis exchange property for rank-2 valuated matroids.
Hence, generalizing this to higher dimensions, \cite{DressTerhalle_padic} introduces the tight span of a rank-$r$ valuated matroid $(E,v)$ as 
\begin{align}
\label{def:TightSpan}
T_{(E,v)}=\Big\{p\in\R^E\ \Big|\ \forall e\in E: p(e) =\max_{e_2,\ldots,e_r \in E}\big\{v(e,e_2,\ldots,e_r)-\sum_{i=2}^r p(e_i)\big\}\Big\}.
\end{align}
The maximum in Equation~\eqref{def:TightSpan} says that the functions $p$ efficiently satisfy a triangle inequality.
Their formulation using $\max$ is dual to our work using $\min$.

In parallel to Proposition~\ref{prop_section_local}, there is a local description of $T_{(E,v)}$ indexed by bases. 
Given a basis $B = \{b_1, \dots, b_r \}$, the map $\Phi_{B}$ that sends a point $(u_1, \dots, u_r)$ in the hyperplane $H_{v(B)} =  \big\{ \sum_{i=0}^r u_i = v(B) \big\}$ to the linear map $\Phi_B(u) : E\rightarrow \R$ given by
\begin{equation*} e \longmapsto \max_{i \in B} \big\{ v(e \cup B \setminus i) + u_i \big\} -  v(B) \end{equation*}
is injective \cite[Proposition 1]{DressTerhalle_treeoflife}. 
There is a similar polyhedral description for intersections $\Phi_A(H) \cap \Phi_B(H)$ with $A$ and $B$ bases of $(E,v)$.
Moreover, as $B$ varies over all bases of $(E,v)$, the whole $T_{(E,v)}$ is covered.
It can be shown via the theory of $(B,N)$ pairs that the $\Phi_B(H)$ form the apartments of a building:

\begin{theorem}[Theorem~1 in \cite{DressTerhalle_treeoflife}]
Let $K$ be a non-archimedean field with discrete valuation, and $w_{\univ} $ the universal realizable matroid of rank-$r$ as in Example~\ref{realizableValuatedMatroids}.
The space $T_{(K^r \setminus 0, w_{\univ})}$ is a geometric realization of the affine building associated to $\GL_r(K)$.
\end{theorem}

A point $p \in T_{(K^r \setminus 0, w_{\univ})}$ given by $p = \Phi_B(u)$ corresponds to seminorm $\seminorm{\cdot}_p : K^r \setminus 0 \to \RR$  given  by
$\seminorm{\cdot}_p = \exp{p(\cdot)}$, which by Lemma~\ref{lem:Cramer} is diagonalizable by $B$ and $u$.
Moreover, $\Phi_B(u)(e) - v(B)$ attains the minimum at least twice for all $e \in E \setminus B$, i.e.~once for $e$ and once for some $i$ in $B$.
Thus, the equations from Definition~\ref{definition_tropicallinearspace} are satisfied for all $\tau = e \cup B$. 
It is straightforward to show that these equations imply the same result for arbitrary $\tau$.
Thus, as remarked by Speyer on \cite[p.6]{Speyer}, the tight span $T_{(E,v)}$ is a lift of $\calL(v)$ to $\R^E$.




\bibliographystyle{amsalpha}
\bibliography{biblio}{}





\end{document}